\documentclass[reqno]{amsart}

    \usepackage{float}
    \usepackage{amsmath}
    \usepackage{graphicx}
    \usepackage{latexsym}
    \usepackage{amsfonts} 
    \usepackage{amssymb}
    \usepackage{verbatim}
    \usepackage{color}
    \usepackage{cleveref}
    \theoremstyle{plain}
    \newtheorem{theorem}{Theorem}
    \newtheorem{corollary}[theorem]{Corollary}
    \newtheorem{lemma}[theorem]{Lemma}
    \newtheorem{proposition}[theorem]{Proposition}
    \theoremstyle{definition}
    \newtheorem{assumption}{Assumption}

    \newtheorem{example}[theorem]{Example}

    \newtheorem{remark}[theorem]{Remark}
    \newtheorem*{remark*}{Remark}
    \newenvironment{assump}[2][]
    {\renewcommand\theassumption{#2}\begin{assumption}[#1]}
        {\end{assumption}}

    \newcommand{\rd}{\mathbb R^d}
    \newcommand{\R}{\mathbb R}
    \newcommand{\Z}{\mathbb Z}
    \newcommand{\pr}{\mathbf P}
    \newcommand{\e}{\mathbf E}

\begin{document}
    \title[Random walks in cones revisited]
    {Random walks in cones revisited} 
    \thanks{
        This research was partially supported  by the Ministry of Science and Higher Education of the Russian Federation, agreement 075-15-2019-1620 date 08/11/2019.
        D. Denisov was supported by a Leverhulme Trust Research Project Grant  RPG-2021-105. 
        V. Wachtel was partially supported by DFG.
    }
    \author[Denisov]{Denis Denisov}
    \address{Department of Mathematics, University of Manchester, UK}
    \email{denis.denisov@manchester.ac.uk}
    
    \author[Wachtel]{Vitali Wachtel}
    \address{Faculty of Mathematics, Bielefeld University, Germany}
    \email{wachtel@math.uni-bielefeld.de}
    
    \begin{abstract}
In this paper we continue 
our study of a multidimensional random walk 
with zero mean  and finite variance  
killed on  leaving a cone. 
We suggest a new approach that allows one to construct 
a positive harmonic function in Lipschitz cones 
under minimal moment conditions. 
This approach allows also to obtain  more accurate 
information about the behaviour of the harmonic function 
not far from the boundary of the cone.  
We also prove limit theorems under new moment conditions.   
\end{abstract}
    
    
    \maketitle

    
    \section{Introduction, main results and discussion}
    
    \subsection{Notation and assumptions.}
    Consider a random walk $\{S(n),n\geq1\}$ on $\rd$, $d\geq1$, where
    $$
    S(n)=X(1)+\cdots+X(n)
    $$
    and $\{X(n), n\geq1\}$ is a family of independent copies of a random
    vector $X=(X_1,X_2,\ldots,X_d)$. Denote by $\mathbb{S}^{d-1}$ the
    unit sphere in $\rd$ centred at the origin and by $\Sigma$ an open and connected subset of
    $\mathbb{S}^{d-1}$. Let $K$ be the cone generated by the rays
    emanating from the origin and passing through $\Sigma$, i.e.
    $\Sigma=K\cap \mathbb{S}^{d-1}$.
    
    Let $\tau_x$ be the exit time from $K$ of the random walk with
    starting point $x\in K$, that is,
    $$
    \tau_x=\inf\{n\ge 1: x+S(n)\notin K\}.
    $$
    In \cite{DW15, DW19} we studied asymptotics for
    $$
    \mathbf P(\tau_x>n),\quad n\to \infty,
    $$
    constructed 
    a positive harmonic function for $S(n)$ killed at leaving $K$  
    and prove conditional limit theorems for this random walk. 
    An important role in our approach was played by the harmonic function of the Brownian motion killed at the
    boundary of $K$ which can be described as the unique  (up to a constant factor), strictly positive  on $K$ solution of the
    following boundary problem:
    $$
    \Delta u(x)=0,\ x \in K\quad\text{with boundary condition }u\big|_{\partial
      K}=0.
    $$

    This function is homogeneous of a certain order $p>0$, that is 
    $u(x) = |x|^p u(x/|x|), x\in K$. 
    The function $u(x)$ and the constant $p$ can be found as follows.
    When $d=1$ then there are  only two non-trivial cones: 
    $K=(0,\infty)$ and $K=(-\infty,0)$.
    For  these cones the harmonic function is given by  $u(x)=|x|$ and, clearly, $p=1$. 
    Assume now that $d\geq2$.
    Let $L_{\mathbb{S}^{d-1}}$ be the Laplace-Beltrami operator on
    $\mathbb{S}^{d-1}$ and assume that $\Sigma$ is regular with respect to $L_{\mathbb{S}^{d-1}}$.
    Under this assumption, there exists a complete set of orthonormal eigenfunctions
    $m_j$  and corresponding eigenvalues $0<\lambda_1<\lambda_2\le\lambda_3\le\ldots$ satisfying
    \begin{align}
     \label{eq.eigen}
      L_{\mathbb{S}^{d-1}}m_j(\sigma)&=-\lambda_jm_j(\sigma),\quad \sigma\in \Sigma\\
      \nonumber m_j(\sigma)&=0, \quad \theta\in \sigma \Sigma. 
    \end{align}
    Then
    \[
    p=\sqrt{\lambda_1+(d/2-1)^2}-(d/2-1)>0
\]
    and the positive harmonic function $u(x)$ of the Brownian motion is given by
    \begin{equation}
    \label{u.from.m}
    u(x)=|x|^pm_1\left(\frac{x}{|x|}\right),\quad x\in K.
    \end{equation}
    Note that~\eqref{u.from.m} implies that 
    \begin{equation}\label{eq:u.bound.main}
        u(x)\le C|x|^p. 
    \end{equation}    
    We refer to~\cite{BS97,DeB87} for some further details 
    about the function $u$ and on the properties of the exit times of Brownian motion
    from the cone $K$. 
    
    If $S(n)$ is a one-dimensional random walk with zero mean and finite variance and if $K=(0,\infty)$ 
    then the function $V(x):=\e[-S(\tau_x)]$ is harmonic for the random walk
    killed at leaving $K$. Typically one proves the finiteness of
    $\e[-S(\tau_x)]$ via the Wiener-Hopf factorisation. 
    In the Appendix of the present paper we give an alternative proof of this fact by constructing an appropriate supermartingale. 
    This is a simplified version of the approach used in this paper 
    for construction of a positive harmonic function. 
    It is worth mentioning that if $S(n)$ is a one-dimensional oscillating random walk then, without further restrictions on its increments, the renewal function of weak descending ladder heights is harmonic.
    
        In~\cite{DW15} (for a particular case of Weyl chambers see~\cite{DW10}) we were considering cones for which the function $u$ can be extended to a harmonic function in a bigger cone.
Under this rather restrictive assumption on $K$ we  showed that if
    $\e[|X|^{p\vee(2+\varepsilon)}]$ is finite 
    for some $\varepsilon>0$ 
    then the limit
    $$
    V(x)=\lim_{n\to\infty}\e[u(x+S(n));\tau_x>n]
    $$
    is finite for all $x\in K$ and  this function is harmonic for the random walk killed at leaving 
    $K$, that is 
    $$
    V(x)=\e[V(x+S(1));\tau_x>1],\quad x\in K.
    $$
    We  also showed that 
    $$
    \lim_{r\to\infty}\frac{V(r\sigma)}{u(r\sigma)}=1
    $$
    for every $\sigma\in\Sigma$. This relation does not give us any information on the behaviour of 
    $V$ close to the boundary of the cone $K$.
    
        Next in~\cite{DW19} we  suggested two alternative constructions of $V$. These new constructions allow one to weaken the geometric restrictions on $K$: it suffices to assume that $K$ is either
    convex or $C^2$ and star-like. A further advantage of these constructions is the estimate
    $$
    |V(x)-u(x)|\le C\left(1+\frac{|x|^{p}}{(d(x))^\gamma}\right),
    $$
    where $d(x)=\mathrm{dist}(x,\partial K)$ and $\gamma$ is a sufficiently small positive number. This implies that $V(x)\sim u(x)$ if $x\to\infty$
    in a such way that $d(x)\ge |x|^{1-\delta}$ for a sufficiently small
    $\delta>0$.
    
    The purpose of the present paper is twofold.
    First, we will  relax the moment assumption $\e[|X|^{p\vee(2+\varepsilon)}]<\infty$. It is worth mentioning that if $p>2$ then, as it has been noticed in \cite{DW15}, the assumption $\e[|X|^p]<\infty$ is optimal. Therefore, the question is whether one can replace $\e[X^{2+\varepsilon}]<\infty$ by a weaker condition in the case when $p\le2$.
    Second, we will analyze the behaviour of the harmonic function $V(x)$
    for all $x$ such that $d(x)\to\infty$.
    
    In this paper we use supermartingale to construct harmonic functions. 
    This approach similar to the approach in~\cite{Var99, Var00,Var01},  
    who used and submartingales to obtain heat kernel estimates for 
    random walks in cones and Lipschitz domains satisfying rather strong moment assumptions.

    \subsection{Main result}

    It turns out that if one wishes to impose weaker 
    moment assumptions then more restrictive smoothness conditions on $K$ are required. 
    Throughout we will impose  
    \begin{assump}{(G)}\label{ass-g}
        Cone $K$ is Lipschitz and star-like.  
        There exists a constant $C>0$ such that 
        \begin{align}\label{eq:harmonic.boundary}
               u(x) &\le  C|x|^{p-1}d(x),\quad x\in K\\ 
            \label{eq:harmonic.boundary.lower}
                  u(x) &\ge C|x|^{p-1}d(x),\quad x\in K. 
        \end{align}        
    \end{assump}    
    \begin{remark} 
    Assumption~\ref{ass-g} holds when 
     $\Sigma$ is $C^{1,\alpha}$ for 
    $\alpha\in(0,1]$, see~\cite{GT2001} for the definition of $C^{1,\alpha}$. 
     It holds even in a slightly  more general case 
     when $\Sigma$ is a Lyapunov-Dini surface, as can be seen 
     from~\cite{Widman1967}.  
     To prove~\eqref{eq:harmonic.boundary} we can consider 
     first the bounded region $D=\{x\in K: 0.5<|x|<2\}$. 
     Then the bound  $u(x)\le Cd(x)$ follows 
     from~\cite[Theorem 2.4]{Widman1967}. For that 
     we need to notice that the bound in this result is proved locally. 
     Since $u=0$ is continuously differentiable 
     at the boundary $\{x\in\partial K\cap: 0.75<|x|<1.25|x| \}$. 
     Using the homogeneity of $u$ we can extend the bound to the whole cone.
     Similarly the lower bound follows from~\cite[Theorem 2.5]{Widman1967} 
     by noticing that this theorem holds locally as well 
     and implies immediately the lower bound for the Green function in the region $D$.  
     Then, the Boundary Harnack Principle implies the bound for harmonic functions. 
     Finally we use the homogeneity of $u$ to extend the bound to the whole cone.
    \end{remark}   
    We will also need the following moment assumptions. 
    \begin{assump}{(M)}\label{ass-m}\ 
        \begin{enumerate}
            \item[(M1)]\label{ass-M1} $\e[X_i]=0, i=1,\ldots d$;
            \item[(M2)]\label{ass-M2}  $\e[X_i^2]=1, i=1,\ldots d$ and $\e[X_iX_j]=0,1\le i<j\le d$;
            \item[(M3)]\label{ass-M3} $\e[|X|^2\log(1+|X|)]<\infty$. In the case $p>2$ we additionally assume that $\e[|X|^p]$  is finite. 
        \end{enumerate}   
    \end{assump}


    \begin{theorem}
    \label{thm:C2}
    Let the assumptions~\ref{ass-g} and~\ref{ass-m} hold. 
    Then the function 
    \begin{equation}
    \label{eq:def.V}
    V(x):=\lim_{n\to\infty}\mathbf{E}\left[u(x+S(n));\tau_x>n\right]
    \end{equation}
    is finite and harmonic for $\{S(n)\}$ killed at leaving $K$, i.e.,
    $$
    V(x)=\mathbf{E}\left[V(x+S(n));\tau_x>n\right],\quad x\in K,\ n\ge1.
    $$
    Furthermore, if $p\ge1$ then 
    $$
    V(x)\sim u(x)\quad\text{for }x\in K\ \text{with } d(x)\to\infty,
    $$
    and 
    \[
        \sup_{x\in K: |x|=o(n^{p/(2(p-1))})}
        \left| 
        \frac{\e \left[u(x+S(n));\tau_x>n\right] - V(x)}{1+u(x)}
        \right| 
        \to 0,\quad  n\to \infty. 
    \]
    If $p<1$ then 
    \[
    V(x)\sim u(x)\quad\text{for }x\in K\ \text{when } d(x)|x|^{p-1}\to\infty.  
    \] 
    
    \end{theorem}
    In our earlier papers \cite{DW15,DW19} it has been shown that 
    the function $V(x)$ from \eqref{eq:def.V} is well defined  
    under a slightly stronger moment condition:
    $\e[|X|^{2+\delta}]<\infty$ for some $\delta>0$ instead of $\e[|X|^2\log(1+|X|)]<\infty$. 
    We will illustrate in examples below that the latter condition is optimal in some sense. 
    On the other hand, in Theorem~\ref{thm:C2} we impose different 
    geometric restrictions on the cone $K$. Namely, we replace $C^2$ assumption on the cone 
    by a milder Assumption~\ref{ass-g}. However the alternative convexity 
    assumption seems to require additional moments and is not covered by 
    Theorem~\ref{thm:C2}. 
    An advantage of the current approach is that it allows one to extend the 
    arguments to the case of Markov chains in a more straightforward way, 
    see~\cite{DZ22}.  
    Construction and behaviour of harmonic functions of Markov chains were considered various situations 
    by many authors, see~\cite{DKW, DW15b,GLL16,GLP16,MS19} and references 
    therein for some of publications.  
    Another improvement of Theorem~\ref{thm:C2} over~\cite{DW15,DW19} is 
    a more accurate information about the asymptotics behaviour of $V(x)$ near the boundary 
    of the cone. 
\begin{theorem}
\label{thm:smooth_and_convex} A
Let the assumptions~\ref{ass-g} and~\ref{ass-m} hold. 
Assume also that  $p\ge1$. 
Then one has: 
\begin{itemize}
 \item[(a)]there exist positive constants $C$ and $R$ such that
 \begin{equation}
 \label{eq:s_and_c.1}
 \pr(\tau_x>n)\le C\frac{u(x+Rx_0)}{n^{p/2}}
 \quad\text{for all }x\in K;
 \end{equation}
 \item[(b)] uniformly in  $x\in K$ such that $|x|\le \frac{\sqrt{n}}{\log n}$, 
 \begin{equation}
 \label{eq:s_and_c.2}
 \pr(\tau_x>n)\sim \varkappa\frac{V(x)}{n^{p/2}}, 
 \quad n\to\infty,  
 \end{equation}
 and
 \begin{equation}
 \label{eq:s_and_c.3}
 \pr\left(\frac{x+S(n)}{\sqrt{n}}\in D\Big|\tau_x>n\right)
 \rightarrow c\int_D u(z)e^{-|z|^2/2}dz
 , \quad n\to\infty,  
 \end{equation}
 for every compact $D\subset K$.
\end{itemize}
\end{theorem}


\subsection{Discussion of the $|X|^2\log(1+|X|)$-condition.}
In this section we shall consider some specific examples of random walks, which will show that the condition $\e[|X|^2\log(1+|X|)]<\infty$ is rather close to the minimal one.
We restrict our attention to the cone
$$
K=\left\{x\in\R^2:\, |x_2|<x_1\right\},
$$
which is a $2$-dimensional Weyl chamber of type $D$. The harmonic function of the Brownian motion killed at leaving $K$ is easy to compute:
$$
u(x)=x_1^2-x_2^2.
$$
Therefore, $p=2$ for this cone.

Let $W=(W_1,W_2)$ be a random vector with the following distribution:
\begin{align}
&\pr(W_1=0,W_2=1)=\pr(W_1=0,W_2=-1)=\frac{1}{4},\\
&\pr(W_1=k,W_2=0)=p_k,\ k=-1,0,1,2,\ldots,
\end{align}
where the numbers $\{p_k\}$ are such that 
$$
\sum_{k=-1}^\infty p_k=\frac{1}{2},\quad 
\sum_{k=-1}^\infty kp_k=0\quad\text{and}\quad 
\sum_{k=-1}^\infty k^2p_k=\frac{1}{2}.
$$
Then one has the equalities 
\begin{equation}
\label{eq:prop.W}
\e[W_1]=\e[W_2]=\e[W_1W_2]=0
\quad\text{and}\quad 
\e[W_1^2]=\e[W_2^2]=\frac{1}{2}.
\end{equation}

\begin{example}
\label{Example1}
We assume that the increments of $S(n)$ are independent copies of the vector $\sqrt{2}W$. It follows from \eqref{eq:prop.W}
that the vector $\sqrt{2}W$ has zero mean and that its covariance matrix is equal to the identity matrix. This implies that the sequence 
$u(x+S(n))$ is a martingale. Furthermore, it is immediate from the definition of $W$ that 
$$
u(x+S(\tau_x))=0
\quad\text{for every}\quad
x\in (\sqrt{2}\Z^2)\cap K.
$$
Using this observation, we obtain 
\begin{align*}
&\e[u(x+S(n)){\rm 1}\{\tau_x>n\}|\mathcal{F}_{n-1}]\\
&\hspace{1cm}=\e[u(x+S(n)){\rm 1}\{\tau_x>n-1\}
-u(x+S(n)){\rm 1}\{\tau_x=n\}|\mathcal{F}_{n-1}]\\
&\hspace{1cm}={\rm 1}\{\tau_x>n-1\}\e[u(x+S(n))|\mathcal{F}_{n-1}]
-\e[u(x+S(\tau_x)){\rm 1}\{\tau_x=n\}|\mathcal{F}_{n-1}]\\
&\hspace{1cm}=u(x+S(n-1)){\rm 1}\{\tau_x>n-1\}.
\end{align*}
In particular,
$$
u(x)=\e[u(x+S(1));\tau_x>1], \quad x\in (\sqrt{2}\Z^2)\cap K.
$$
Thus, $u(x)$ is harmonic for $S(n)$ killed at leaving $K$.
So, a harmonic function may exist without further moment restrictions. 
We now show that the relation 
\begin{equation}
\label{ex.1}
\pr(\tau_x>n)\sim\varkappa\frac{u(x)}{n}
\end{equation}
may fail in this case. To this end we assume that the numbers $\{p_k\}$ are such that 
\begin{equation}
\label{ex.2}
\log m \e[W_1^2;W_1\ge m]\to\infty
\quad\text{as}\quad m\to\infty.
\end{equation}
Consider now the stopping time 
$$
\sigma_n:=\inf\{k\ge1:\, X_1(k)\ge 2n^2\}.
$$
It is easy to see that if $\tau_x>\sigma_n=j$ then 
$$
S_1(k)\ge X_1(j)-\sqrt{2}(k-j)
\quad\text{and}\quad
|S_2(k)|\le\sqrt{2}k
\quad\text{for all }\quad k\ge j.
$$
In particular,
$$
\{\tau_x>n,\sigma_n=j\}=\{\tau_x>j-1,\sigma_n=j\}
$$
and, for all $n$ large enough, 
$$
u(x+S(n))\ge X_1^2(j)-2(|x|+\sqrt{2}n)^2\ge\frac{X_1^2(j)}{2}
\quad\text{on the event}\quad
\{\tau_x>j-1, \sigma_n=j\}.
$$
Therefore,
\begin{align*}
\e[u(x+S(n));\tau_x>n]
&\ge \e[u(x+S(n));\tau_x>n,\sigma_n\le n]\\
&\ge\sum_{j=1}^{n}\pr(\tau_x>j-1,\sigma_n>j-1)
\frac{1}{2}\e[X_1^2(j);X_1(j)\ge 2n^2]\\
&=\e[W_1^2;W_1\ge \sqrt 2 n^2]\sum_{j=1}^{n}\pr(\tau_x>j-1,\sigma_n>j-1).
\end{align*}
Furthermore, for every $j\le n$ one has 
\begin{align*}
\pr(\tau_x>j-1,\sigma_n>j-1)
&\ge \pr(\tau_x>j-1)-\pr(\sigma_n\le j-1)\\
&\ge \pr(\tau_x>j-1)-(j-1)\pr(W_1\ge n^2).
\end{align*}
Therefore,
\begin{align}
\label{ex.3}
\nonumber
\e[u(x+S(n));\tau_x>n]
&\ge \e[W_1^2;W_1\ge n^2]\sum_{j=1}^{n}\pr(\tau_x>j-1)
-n^2\pr(W_1\ge n^2)\\
&= \e[W_1^2;W_1\ge n^2]\sum_{j=1}^{n}\pr(\tau_x>j-1)+o(1).
\end{align}
If \eqref{ex.1} holds then 
$$
\sum_{j=1}^{n}\pr(\tau_x>j-1)\sim \varkappa u(x)\log n.
$$
Plugging this into \eqref{ex.3} and taking into account
\eqref{ex.2}, we conclude that 
$$
\e[u(x+S(n));\tau_x>n]\to\infty.
$$
This contradicts to the harmonicity of $u(x)$. Consequently, \eqref{ex.1} can not hold for walks satisfying \eqref{ex.2}. 
\hfill$\diamond$
\end{example}

We next show that the limit of the sequence $\e[u(x+S(n));\tau_x>n]$
may be infinite.\
\begin{example}
\label{Example2}
Now we assume that the increments of $S(n)$ are independent copies of the vector $\sqrt{2}(W_2,W_1)$. In this case we have 
\begin{equation}
\label{ex.4}
u(x+S(\tau_x))\le0
\quad\text{for all}\quad 
x\in(\sqrt{2}\Z^2)\cap K.
\end{equation}
By the optional stopping theorem for the martingale $u(x+S(n))$,
\begin{align*}
u(x)
&=\e[u(x+S(\tau_x\wedge n))]\\
&=\e[u(x+S(n));\tau_x>n]+\e[u(x+S(\tau_x));\tau_x\le n].
\end{align*}
Consequently,
$$
\e[u(x+S(n));\tau_x>n]=u(x)-\e[u(x+S(\tau_x));\tau_x\le n].
$$
\eqref{ex.4} allows one to apply the monotone convergence theorem.
As a result we have 
\begin{equation}
\label{ex.5}
\lim_{n\to\infty}\e[u(x+S(n));\tau_x>n]
=u(x)-\e[u(x+S(\tau_x))].
\end{equation}

For every $y\in K$ we have 
\begin{align*}
\e[-u(y+X);y+X\notin K]
&=\e[(y_2+X_2)^2-(y_1+X_1)^2;y+X\notin K]\\
&=\e[(y_2+X_2)^2-(y_1+X_1)^2;X_2>y_1-y_2]\\
&=\e[(y_2+X_2)^2-y_1^2;X_2>y_1-y_2]\\
&\ge\frac{1}{2}\e[(y_2+X_2)^2;X_2>\sqrt{2}y_1-y_2]\\
&\ge\frac{1}{8}\e[X_2^2;X_2>2(|y_1|+|y_2|)].
\end{align*}
Therefore,
\begin{align}\label{ex.5a}
\nonumber
&\e[-u(x+S(\tau_x))]\\
\nonumber
&=\sum_{n=0}^\infty\int_K\pr(x+S(n)\in dy;\tau_x>n)
\e[-u(y+X);y+X\notin K]\\
\nonumber
&\ge \frac{1}{8}\sum_{n=0}^\infty\int_K\pr(x+S(n)\in dy;\tau_x>n)
\e[X_2^2;X_2>2(|y_1|+|y_2|)]\\
\nonumber
&= \frac{1}{4}\sum_{n=0}^\infty\int_K\pr(x+S(n)\in dy;\tau_x>n)
\e[W_1^2;W_1>\sqrt{2}(|y_1|+|y_2|)]\\
&\ge \frac{1}{4}
\sum_{n=0}^\infty \pr(|x_1+S_1(n)|+|x_2+S_2(n)|\le n^2;\tau_x>n)\e[W_1^2;W_1>\sqrt{2}n^2].
\end{align}
It is immediate from the definition of the vector $W$ that 
$$
|x_2+S_2(n)|<x_1+S_1(n)\le x_1+\sqrt{2}n
$$
on the event $\{\tau_x>n\}$. This yields the existence of $n_0=n_0(x)$ such that 
\begin{equation}
\label{ex.5b}
\pr(|x_1+S_1(n)|+|x_2+S_2(n)|\le n^2;\tau_x>n)
=\pr(\tau_x>n),\quad n\ge n_0.
\end{equation}
If \eqref{ex.2} is valid and 
$$
\pr(\tau_x>n)\ge \frac{c(x)}{n}
$$
then
$$
\e[-u(x+S(\tau_x))]=\infty.
$$
Combining this with \eqref{ex.5}, we conclude that 
$$
\lim_{n\to\infty}\e[u(x+S(n));\tau_x>n]=\infty.
$$
\hfill$\diamond$ 
\end{example}
It turns out that if the increments of the walk are independent copies of 
of the vector $\sqrt{2}(W_2,W_1)$ then the asymptotic behaviour of
$\pr(\tau_x>n)$ can be studied without constructing a positive harmonic function.
\begin{proposition}
\label{prop:example2}
For the random walk from Example~\ref{Example2} we have:
\begin{itemize}
 \item[(a)] the sequence $E_n:=\e[u(x+S(n));\tau_x>n]$ is monotone increasing and slowly varying;
 \item[(b)] as $n\to\infty$,
            $$
            \pr(\tau_x>n)\sim\varkappa\frac{E_n}{n};
            $$ 
 \item[(c)] the limit of $E_n$ is finite if and only if $\e[W_1^2\log(1+W_1)]$ is finite.
\end{itemize}
\end{proposition}
\section{Preliminary estimates}    
In this section we will present some preliminary bounds. 

\subsection{Estimates for the harmonic and Green functions}
Here, we collect some information that will be used in the sequel. 

Recall that  $d(x) = {\rm dist}(x,\partial K)$.  
and note that following  simple  bound
\[
    d(x)\le |x|,\quad x\in K. 
\]
We will use the standard multi-index notation for partial derivatives, 
that is for $\alpha=(\alpha_1,\ldots,\alpha_d)$ we put 
$|\alpha| = \alpha_1+\cdots+\alpha_d$ and 
\[
\frac{\partial^\alpha f(x)} {\partial x_\alpha}   
 = \frac{\partial^{|\alpha|} f(x)} {\partial x_{\alpha_1}\ldots \partial x_{\alpha_d}}.   
\]

We shall always assume that $u(x)=0$ for all $x\notin K$.

 We will make use of the following   result, see~\cite[Lemma 2.1]{DW19}. 
 \begin{lemma}
    \label{lem:harnack}
    There exists a constant $C=C(d)$ such that 
    for $x\in K$ and  $\alpha$ with $|\alpha|\le 3$, 
    \begin{align}
    \left|\frac{\partial^\alpha u(x)} {\partial x^\alpha} \right|  
    \le C\frac{u(x)}{d(x)^{|\alpha|}}
    \label{eq:bound.u}
    \end{align}
\end{lemma}
We will also make use of the following result proved in~\cite[Lemma~2.3]{DW19}. 
\begin{lemma}\label{lem:u-diff}
  Assume that equation~\eqref{eq:harmonic.boundary} holds. 
  Let $x\in K$. Then, 
  \begin{equation}
    \label{diff-bound}
    |u(x+y)-u(x)|\le  C|y|\left(|x|^{p-1}+|y|^{p-1}\right)
  \end{equation}
and, for $|y|\le |x|/2$,
\begin{equation}
    \label{diff-bound1}
    |u(x+y)-u(x)|\le  C|y||x|^{p-1}.
  \end{equation}
For $p<1$ and $x\in K$, 
\begin{equation}
  \label{diff-bound2}
  |u(x+y)-u(x)|\le  C|y|^{p}. 
\end{equation}
 \end{lemma}  
 This lemma has been formulated under slightly different assumptions. 
 However, it can be seen from the proof that assumption~\eqref{eq:harmonic.boundary} is sufficient.

Let $G(x,y)$ be the Green function of Brownian motion 
killed on leaving  the cone $K$.  
For all $y\in K$ this function is harmonic in $x\neq y$ and vanishes at the  boundary of the cone. 

\begin{lemma}\label{lem:green.bound0}
    Let the assumption~\ref{ass-g} holds. 
    Then, for any $A>0$ there exists   a constant $C_A$ such that 
\begin{equation}
 \label{eq:green.function.bound0}
 G(x,y)  \le C_A \widehat G(x,y), 
 \end{equation} 
where 
\begin{equation}\label{eq:ghat}
    \widehat G(x,y)  := 
\begin{cases}
    \frac{u(x)u(y)}{|y|^{d-2+2p}},& |x|\le |y|, |x-y|\ge A|y|\\
     \frac{u(x)u(y)}{|x|^{d-2+2p}},& |y|\le |x|, |x-y|\ge A|y|\\ 
     \frac{u(x)u(y)}{|x|^{p-1}|y|^{p-1}} \frac{1}{|x-y|^d}
                                   &d(y)/2\le |x-y| \le A|y|\\
                                   \frac{1}{|x-y|^{d-2}}I(d>2)
                                   +\ln\left(\frac{d(y)}{|x-y|}\right)I(d=2)
                                   & |x-y|<d(y)/2.
 \end{cases}
 \end{equation}
 \begin{proof}
     When $\Sigma $ is of class $C^2$ the required bound was proved in~\cite[Lemma~1 and Lemma~4]{Azarin66}. 
    The result was derived from~\cite{KM39}(see also~\cite[Theorem 2.3]{Widman1967}) 
    by transferring an estimate 
    \[
      G(x,y)\le \frac{d(x)d(y)}{|x-y|^d}    
    \]
    for the Green function 
    proved for bounded Lyapunov domains to an estimate for cones. 
    It is clear that  the same proof will work provided condition~\eqref{eq:harmonic.boundary.lower} holds.  

    Alternatively, for $d\ge 3$ 
    the  bound follows from the following result proved in~\cite[Remark 3.1]{H06}, 
     \[
G(x,y)  \le C\frac{u(x)u(y)}{\max(|x|,|y|)^{2p-2}|x-y|^d}.
     \]     
     Indeed when $\Sigma $ is Lipschitz 
     the cone $K$ is uniform.  
     Then the result follows from the assumptions~\eqref{eq:harmonic.boundary} and~\eqref{eq:harmonic.boundary.lower}. 
     The case $d=2$ follows from~\cite[Lemma~1 and Lemma~4]{Azarin66}. 
     For $d=2$ the last line in the bound is an estimate of the Green function for cone $K$ 
     can be derived  via the Green function for the ball $B(y,\delta(y)/2)$.  
 \end{proof}    

 \end{lemma}

\subsection{Estimates for the error term}

Let $u(x)=0$ when $x\notin K$ and let 
    \begin{equation}\label{eq:defn.f}
        f(x) = \e[u(x+X)] -u(x).
    \end{equation}
The following lemma will be convenient in the construction to follow. 
\begin{lemma}\label{lem:construction.gamma}
Assume that $\e[|X|^2]<\infty$  and, in addition, 
\[
        \begin{cases}
            \e[|X|^p]<\infty,& p>2;\\ 
            \e[\ln(1+|X|)|X|^2]<\infty,& p=2.
        \end{cases}    
    \]
        Then, there exists a slowly varying, monotone decreasing differentiable function $\gamma(t)$ such that  
        \begin{align}\label{eq:gamma.integral.finite}
            \e[|X|^{p};|X|>t]&=o(\gamma(t)t^{p-2})\\ 
            \label{eq:gamma.integral.finite.3}
            \int_1^\infty x^{-1}\gamma(x) dx&<\infty.
        \end{align}
    \end{lemma}    
\begin{proof}
We consider first the case $p\ge1$. It is clear that 
$$
t\mapsto\frac{\e[|X|^p;|X|>t]}{t^{p-1}}
\quad\text{is monotone decreasing.}
$$
Furthermore, by the Fubini theorem,
\begin{align*}
\int_1^\infty \frac{\e[|X|^p;|X|>t]}{t^{p-1}}dt
&=\e\left[|X|^p\int_1^{|X|}t^{1-p}dt; |X|>1\right]\\
&\le\left\{
\begin{array}{ll}
\frac{\e[|X|^2]}{(2-p)}, &p<2,\\
\e[|X|^2\log(1+|X|)], &p=2,\\
\frac{\e[|X|^p]}{(p-2)}, &p>2.\\
\end{array}
\right.
\end{align*}
Thus we may apply the result of \cite{D06}: there exists a slowly varying function $\ell(t)$ such that $\ell(t)/t$ is integrable and 
$$
\frac{\e[|X|^p;|X|>t]}{t^{p-1}}\le\frac{\ell(t)}{t}.
$$
Equivalently,
$$
\e[|X|^p;|X|>t]\le \ell(t)t^{p-2}.
$$
It is possible to choose a decreasing, slowly varying function $\gamma(t)$ such that $\gamma(t)/t$ is integrable and $\ell(t)=o(\gamma(t))$. This completes the proof in the case $p\ge1$.

When  $p<1$ then, using the Markov inequality, we get 
\begin{align*}
\frac{\e[|X|^p;|X|>t]}{t^{p-1}}
&=t^{1-p}\e\left[\frac{|X|}{|X|^{1-p}};|X|>t\right]
\le \e[|X|;|X|>t] 
\end{align*}
Now it remains to apply the already proven estimate for $p=1$. 

Finally note that it is not difficult to achieve differentiability of $\gamma$ by 
 adjusting it. 
\end{proof}

    \begin{lemma}\label{lem:bound.f}
        Assume that equation~\eqref{eq:harmonic.boundary} holds   
        and let $f$ be defined by~\eqref{eq:defn.f}.
      Then, 
      \begin{equation*} 
          |f(x)|=o(\beta(x)),\quad d(x)\to\infty,   
        \end{equation*}
    where 
    \begin{equation}\label{defn.beta}
        \beta(x):= |x|^{p-1}\frac{\gamma(d(x))}{d(x)}.
    \end{equation}
    \end{lemma}
    \begin{proof}

    We will start with a Taylor theorem for a thrice differentiable function $U$. 
    Let $y$ be such that $|y|\le d(x)/2$. Then, 
    \begin{equation}\label{eq:taylor}
          \left|U(x+y)-U(x)-\nabla U(x) \cdot y
          -\frac{1}{2}\sum_{i,j}\frac{\partial^2 U(x)}{\partial x_i \partial x_j}y_iy_j\right| \le R_{2,\theta}(x)|y|^{2+\theta},
      \end{equation}
      where 
      \[
          R_{2,\theta}(x)= 
          \sup_{i,j}
          [U_{x_i x_j}]_{\theta, B(x,d(x)/2)}<\infty. 
    \]
Here,  we let for $\theta\in (0,1]$ and open $D$, 
 \[
    [f]_{\theta, D} = 
    \sup_{y,z\in D, y\neq z}  \frac{|f(y)-f(z)|}{|y-z|^\theta}. 
 \]
We can further proceed as follows
    \begin{align*}
      &  \left| \e\left[U(x+X)-U(x)-\frac{1}{2}\Delta U(x)\right]
      \right|\\
      &=\left|\e \left[U(x+X)-U(x);|X|\le d(x)/2\right] -\frac{1}{2}\Delta U(x)\right| \\
      &\hspace{0.5cm}+\left|\e \left[U(x+X)-U(x);|X|> d(x)/2)\right] \right| \\
      &\le\left| \e \left[\left(\nabla U(x)\cdot X +\frac{1}{2}\sum_{i,j}\frac{\partial^2 U(x)}{\partial x_i\partial x_j}X_iX_j\right){\rm 1}(|X|\le d(x)/2)\right]-\frac{1}{2}\Delta U(x)\right|\\
      &\hspace{0.5cm}+R_{2,\theta}(x)\e \left[|X|^{2+\theta};|X|\le d (x)/2\right]
      +\left|\e \left[U(x+X)-U(x);|X|>d(x)/2)\right] \right|. 
    \end{align*} 
    Now rearrange the terms and make use of the assumptions   
    $\e[X_i]=0$, $\mbox{cov}(X_i,X_j)=\delta_{i=j}$  to obtain 
    \begin{align}\label{eq:taylor.stochastic}
      &  \left| \e\left[U(x+X)-u(x)-\frac{1}{2}\Delta U(x)\right]\right|\\ 
      \nonumber 
      &\hspace{1cm} \le 
\left| \e \left[\left(\nabla U(x)\cdot X +\frac{1}{2}\sum_{i,j}\frac{\partial^2 U(x)}{\partial x_i\partial x_j}X_iX_j\right){\rm 1}(|X|>d(x)/2)\right]\right|\\
      \nonumber 
      &\hspace{1cm}+R_{2+\theta}(x)\e \left[|X|^{2+\theta};|X|\le 
      d(x)/2\right]
      +\left|\e \left[U(x+X)-U(x);|X|>d(x)/2\right] \right|. 
    \end{align}
    We will now make use of~\eqref{eq:taylor.stochastic} with $U=u$ 
    and $\theta =1$.  
    Applying the inequalities~\eqref{eq:bound.u}, \eqref{eq:u.bound.main} and~\eqref{eq:harmonic.boundary} 
    one can estimate the remainder  as follows, 
    \[
        |R_{2,\alpha}(x)|\le 
        |x|^{p-1}d(x)^{-2}. 
    \]
    Also, using~\eqref{diff-bound} and~\eqref{diff-bound1}, 
    we can estimate 
    \begin{align*}
      &  |\e \left(u(x+X)-u(x)\right) ;|X|>d(x)) |\\ 
      &\hspace{1cm}\le 
      C\e[|X|^p;|X|>|x|/2)]
      +C|x|^{p-1}\e[|X|;|X|>d(x)/2].
    \end{align*}
    Thus, it follows from~\eqref{eq:taylor.stochastic} that 
    \begin{align*}
      |f(x)|&\le \left|\e\left[\left(\nabla u(x)\cdot X +\frac{1}{2}\sum_{i,j}\frac{\partial^2 u}{\partial x_i\partial x_j}X_iX_j\right){\rm 1}(|X|>d(x)/2)\right]\right|\\
            &\hspace{1cm}+C|x|^{p-1}d(x)^{-2}\e \left[|X|^{3};|X|\le d(x)/2\right]\\
             &\hspace{1cm}+C\e[|X|^p;|X|>|x|/2]
             +C|x|^{p-1}\e[|X|;|X|>d(x)/2].
    \end{align*}
    The partial derivatives of the function $u$
    in the first term can be  estimated via 
    Lemma~\ref{lem:harnack} and~\eqref{eq:harmonic.boundary}, 
    which results in the following estimate 
    \begin{align*}
\left|\e\left[\nabla u(x)\cdot X ;|X|>d(x)/2\right]\right|\le
C|x|^{p-1}\e[|X|;|X|>d(x)/2]. 
    \end{align*}    
    We estimate the terms with the second derivative 
    using~\eqref{eq:bound.u} and~\eqref{eq:harmonic.boundary}, 
    \begin{align*}
\frac{1}{2}\left|\e\left[\sum_{i,j}\frac{\partial^2 u(x)}{\partial x_i\partial x_j}X_iX_j;|X|>d(x)/2\right]\right|\le 
C|x|^{p-1}d(x)^{-1}\e[|X|^2;|X|>d(x)/2]. 
\end{align*}    
Then, 
\begin{align*}
        |f(x)|&\le C\biggl(|x|^{p-1} \e \left[|X|;|X|>
          d(x)/2\right]
          +|x|^{p-1}d(x)^{-1}\e \left[|X|^2;|X|>d(x)/2\right]\\
            &\hspace{1cm}
            +|x|^{p-1}(d(x))^{-2}\e \left[|X|^{3};|X|\le d (x)/2\right]\\
             &\hspace{1cm}+\e[|X|^p;|X|>|x|/2]
             +|x|^{p-1}\e[|X|;|X|>d(x)/2]
      \biggr).
    \end{align*}
    The term with the  moment 
    of the order $3$ 
    can be estimated using~\eqref{eq:gamma.integral.finite} with $p=1$,
    \begin{align*}      
&\e \left[|X|^{3};|X|\le d (x)/2\right]
\le 3 \int_0^{d(x)/2}y^{2}\pr(|X|>y)dy\\
&\hspace{1cm}\le 3 C_A  \int_A^{d(x)/2}y^{2}\frac{\e[|X|;|X|>y]}{y}dy\\
&\hspace{1cm}\le 
C_A+C \int_A^{d(x)/2} \gamma(y) dy
\le C_A + \varepsilon_A d(x) \gamma(d(x)), 
\end{align*}    
where $\varepsilon_A\to 0$ as $A\to \infty$ and 
we used the slow variation of $\gamma$ in the final step. 
Hence, 
\begin{equation}
    \label{eq:third.moment}
    \e \left[|X|^{3};|X|\le d (x)/2\right] = o(1) 
    d(x)\gamma(d(x)),\quad d(x)\to \infty. 
\end{equation}
Then, for $p\ge 2$, using the Markov inequality we can further simplify the right-hand-side as follows, 
    \[
        |f(x)|\le o(1)(\e[|X|^p;|X|>|x|]+|x|^{p-1}\e[|X|;|X|>d(x)/2]
        +|x|^{p-1}\gamma(d(x)))(d(x))^{-1}.
    \]
    Applying Lemma~\ref{lem:construction.gamma} we obtain 
\begin{align*}
    |f(x)| &\le o(1)\left(\gamma(x)|x|^{p-2}+|x|^{p-1}\gamma(d(x))/d(x)
    \right) \\ 
           &\le o(1)|x|^{p-1}\gamma(d(x))/d(x),
\end{align*}
since $d(x)\le |x|$. 
For $p<2$ we have the following estimate 
\[
    |f(x)|\le o(1)|x|^{p-1}\gamma(d(x))/d(x). 
\]

    \end{proof}

\section{Construction of a non-negative supermartingale}

For $x\in K$  let  $\beta(x)$ be the function defined in~\eqref{defn.beta}. 
Let 
\begin{equation}\label{eq:ubeta}
    U_\beta(y) = \int_K G(x,y) \beta(x) dx.  
\end{equation}
By the definition of the Green function 
\begin{equation}\label{eq:laplacian.drift}
    \Delta U(y) = -\beta(y).
\end{equation}

Using~\eqref{eq:green.function.bound0} one can show that $U_\beta$  is well defined.  For that we will estimate the integral in~\eqref{eq:ubeta} in a sequence of Lemmas.
Pick $A<1$ and then   $C_A$ such that 
Lemma~\ref{lem:green.bound0} holds. 
\begin{lemma}\label{lem:ghat.1}
Let the assumption~\ref{ass-g}  hold.  
Then,  there exists a function $\varepsilon(R)\to 0$ such that for $y\in K: |y|>R$,
\begin{equation}\label{eq:ghat.1}
    I_1(y):=\int_{K \cap \{|x|\le |y|, |x-y|\ge A |y|\}}\widehat G(x,y)\beta(x) dx \le \varepsilon(R) u(y).
\end{equation}
\end{lemma}    
\begin{proof}
Put $d_\Sigma(\sigma) = {\rm dist}(\sigma, \partial \Sigma) $. 
Then there exists a constant $c_0>0$ such that
\begin{equation}\label{eq:dist}
d(r \sigma) < r d_\Sigma(\sigma) < c_0 d(r\sigma).   
\end{equation}
We have, using~\eqref{eq:harmonic.boundary}, 
    \begin{align*}
        \frac{I_1(y)}{u(y)}&\le |y|^{2-d-2p}
    \int_{K \cap \{|x|\le |y|\}} 
    u(x) |x|^{p-1}\frac{\gamma(d(x))}{d(x)}dx\\
              &\le C|y|^{2-d-2p}
    \int_{K \cap \{|x|\le |y|\}} 
    |x|^{2p-2} \gamma(d(x))dx\\
              &\le C |y|^{2-d-2p}
              \int_{\Sigma} d\sigma \int_0^{|y|} r^{2p+d-3}
              \gamma(d(r\sigma))dr.
    \end{align*}
 Then we can split the $\Sigma = \Sigma_0\cup\Sigma_1$ in such a way that 
 for $\Sigma_0$ contains all $\sigma\in\Sigma$ with the distance 
 $d_\Sigma(\sigma)\le \varepsilon_0$, where 
 $\varepsilon_0 $ is sufficiently small  to ensure that 
 \begin{multline*} 
C |y|^{2-d-2p}
\int_{\Sigma_0} d\sigma \int_0^{|y|} r^{2p+d-3}
\gamma(d( r\sigma))dr\\ 
\le 
\sup_{t\ge 0}  
\gamma(t)
C |y|^{2-d-2p}
\int_{\Sigma_0} d\sigma \int_0^{|y|} r^{2p+d-3}dr 
<C\mathrm{area}(\Sigma_0)\le 
\varepsilon/2. 
 \end{multline*}
Next  using~\eqref{eq:dist} and regular variation of $\gamma$, 
\begin{align*}
&C |y|^{2-d-2p}
              \int_{\Sigma_1} d\sigma \int_0^{|y|} r^{2p+d-3}
              \gamma( d(r\sigma))dr\\ 
&\hspace{1cm}\le 
C |y|^{2-d-2p}
              \int_0^{|y|} r^{2p+d-3}\gamma(c_0^{-1} r\varepsilon_0)dr\\
&\hspace{1cm}\le 
C \gamma( |y| )\le  
\varepsilon/2,
\end{align*}
for $|y|>R$ and sufficiently large $R$ due to the fact that $\gamma(t)\to 0.$

\end{proof}

\begin{lemma}\label{lem:ghat.2}
Let the assumption~\ref{ass-g}  hold. 
Then,  there exists a function $\varepsilon(R)\to 0$ such that for $y\in K: |y|>R$,
\begin{equation}\label{eq:ghat.2}
    I_2(y):=\int_{K \cap \{|x|\ge |y|, |x-y|\ge A |y|\}}
    \widehat G(x,y)\beta(x) dx \le \varepsilon(R) u(y).
\end{equation}
\end{lemma}    
\begin{proof}
To prove the statement it is sufficient to show that for any $\varepsilon>0$ there exists $R>0$ such that 
$I_2(y)\le \varepsilon u(y)$  for all $y:|y|>R$. 
    We have, using~\eqref{eq:harmonic.boundary}, 
    \begin{align*}
        \frac{I_2(y)}{u(y)}&\le 
        \int_{K \cap \{|x|\ge |y|\}} 
    u(x)|x|^{2-d-2p}
 |x|^{p-1}\frac{\gamma(d(x))}{d(x)}dx\\
              &\le C
    \int_{K \cap \{|x|\ge |y|\}} 
    |x|^{-d} \gamma(d(x))dx\\
              &\le C 
              \int_{\Sigma} d\sigma 
              \int_{|y|}^{\infty} r^{-1}
              \gamma(c_0^{-1}rd_\Sigma(\sigma))dr\\ 
  &\le C  \int_{\Sigma} d\sigma \int_{|y|c_0^{-1}d_\Sigma(\sigma)}^{\infty} r^{-1}\gamma(r )dr.  
    \end{align*}
 Let  $c_1=diam(\Sigma)$. 
 Then,  
 \begin{align*}
     & C \int_{\Sigma} d\sigma 
              \int_{c_0^{-1}|y|d_\Sigma(\sigma)}^{\infty} r^{-1}\gamma(r )dr 
              =
              C \int_0^\infty r^{-1}\gamma(r )
              \int_{\Sigma\cap \{c_0^{-1}d_\Sigma(\sigma)|y|<r\}} d\sigma 
              dr
              \\
      &        \le C 
              \int_{0}^{\infty} r^{-1}\gamma(r )
              \min \left(\frac{rc_0}{|y|}, c_1\right)
              dr 
              \le \varepsilon,
 \end{align*}
 since 
 \[
    \int_{0}^{\infty} r^{-1}\gamma(r )
    \min \left(\frac{rc_0}{|y|}, c_1\right)
    dr 
    =\frac{c_0}{|y|}\int_0^{c_0c_1|y|} \gamma(r) dr 
    +c_1 \int_{c_0c_1|y|}^\infty \gamma(r)r^{-1} dr
 \]
 can be made small 
for $|y|>R$ and sufficiently large $R$ due to the  convergence of the integral
and the fact that $\gamma$ converges to $0$. 
\end{proof}

\begin{lemma}\label{lem:ghat.3}
Let the assumption~\ref{ass-g} hold. 
Then,  there exists a bounded monotone 
function $\varepsilon(R)\to 0$ such that for $y\in K: d(y)>R$,
\begin{equation}\label{eq:ghat.3}
    I_3(y):=
    \int_{K \cap \{d(y)/2<|x-y|<A|y|\}}\widehat G(x,y)\beta(x) dx \le \varepsilon(R) u(y).
\end{equation}
\end{lemma}    
\begin{proof}
First note that $|x-y|<A|y|$ with $A<1$ implies that 
\[
(1-A)|y|<|x|<(1+A)|y|. 
\]    
    We have, 
\begin{align*}
    \frac{I_3(y)}{u(y)}&\le C|y|^{1-p} 
    \int_{K \cap \left\{d(y)/2<|x-y|<A|y|\right\}}
    \frac{u(x)|x|^{1-p}}{|x-y|^d}\beta(x) dx\\
&\le C|y|^{1-p}  \int_{K \cap \left\{d(y)/2<|x-y|<A|y|\right\}}
    \frac{d(x)}{|x-y|^d}|x|^{p-1}\frac{\gamma(d(x))}{d(x)} dx\\
          &\le
    C \int_{K \cap \left\{d(y)/2<|x-y|<A|y|\right\}}
    \frac{\gamma(d(x))}{|x-y|^d} dx. 
\end{align*}    
Splitting the integral into regions we obtain 
\begin{align*}
    \frac{I_3(y)}{u(y)}&\le C\sum_{n=0}^{[\log_2 A|y|/d(y)]}
    \int_{K \cap \left\{2^{n-1}d(y)<|x-y|<
    2^{n}d(y)\right\}}
    \frac{\gamma(d(x))}{|x-y|^d} dx\\
    &\le     
    \sum_{n=0}^{[\log_2 A|y|/d(y)]}
    \frac{C}
    {2^{dn}d(y)^{d}}
    \int_{K \cap \left\{2^{n-1}d(y)<|x-y|<
    2^{n}d(y)\right\}}
    \gamma(d(x))dx\\
    &\le 
    \sum_{n=0}^{[\log_2 A|y|/d(y)]}
    \frac{C}
    {2^{dn}d(y)^{d}}
    (2^nd(y))^{d-1}
    \int_{0}^{2^nd(y)}
    \gamma(r)dr\\
    &\le C
    \sum_{n=0}^{[\log_2 A|y|/d(y)]}
    \gamma(2^nd(y)),
\end{align*}
using the slow variation of $\gamma$ in the last inequality
and the definition of $A$.  
Then, 
\begin{align*}
    \frac{I_3(y)}{u(y)}\le C\int_0^{\infty}
    \gamma(2^{t-1}d(y)) dt 
    \le C \int_{1/2}^{\infty}
    \frac{\gamma(zd(y))}{z} dz 
    =
    C \int_{d(y)/2}^{\infty}
    \frac{\gamma(zd(y))}{z} dz. 
\end{align*}
We obtain immediately that for $y:d(y)>R$, 
\[
    \frac{I_3(y)}{u(y)}\le C 
    \int_{R/2}^{\infty}
    \frac{\gamma(zd(y))}{z} dz, 
\]
which is finite and converges to $0$ as $R\to \infty$. 
\end{proof}    

\begin{lemma}\label{lem:ghat.4}
Let the assumption~\ref{ass-g} 
hold. 
Then,  there exists a bounded monotone decreasing function 
$\varepsilon(R)\to 0$, as $R\to\infty$,  such that for $y\in K: d(y)>R$,
\begin{equation}\label{eq:ghat.4}
    I_4(y):=\int_{K \cap \{|x-y|<d(y)/2\}}\widehat G(x,y)\beta(x) dx \le \varepsilon(R) u(y).
\end{equation}
\end{lemma}    

\begin{proof}
We consider the case $d\ge 3$, as the case $d=2$ is similar. 
For $|x-y|\le d(y)/2$ we  use the bound 
\begin{align*}
    I_{4}(y) &\le C
\int_{K \cap \{|x-y|\le d(y)/2\}}
\frac{1}{|x-y|^{d-2}}\beta(x) dx\\ 
&\le 
C
\int_{K \cap \{|x-y|\le d(y)/2\}}
\frac{|x|^{p-1}}{|x-y|^{d-2}}\frac{\gamma(d(x))}{d(x)} dx
\end{align*}
Note that $|x-y|\le d(y)/2$ implies that
$\frac{|y|}{2}\le |x|\le\frac{3|y|}{2}$ and
$\frac{d(y)}{2}\le d(x)\le\frac{3d(y)}{2}$.
Using these estimates and~\eqref{eq:harmonic.boundary.lower} we obtain 
\[
\frac{I_4(y)}
{u(y)}\le 
\frac{C\gamma(d(y))}{d(y)^2}
\int_{ \{|x-y|\le d(y)/2\}} \frac{dx}{|x-y|^{d-2}}   
\le C\gamma(d(y)),
\]
implying the statement of the theorem since 
$\gamma$ is monotone. 
\end{proof}

\begin{lemma}\label{lem:ubeta.bound0}
Let the assumption~\ref{ass-g} hold. 
Then, $U_\beta(y)$ is finite. 
Moreover,  there exists a bounded and monotone 
function $\varepsilon(R)\to 0$, as $R\to \infty$, such that for $y\in K: d(y)>R$, 
     the following estimate is valid
     \begin{equation}\label{eq:ubeta.bound0}
         U_\beta(y)\le \varepsilon(R) u(y).
     \end{equation}    
\end{lemma}    
\begin{proof}
    The statement follows from the bound $G(x,y)\le C\widehat G(x,y)$ and 
    Lemma~\ref{lem:ghat.1}--Lemma~\ref{lem:ghat.4}. 
\end{proof}

We will also need estimates for the derivatives $U_\beta$. 
\begin{lemma}\label{lem:drift.main.term}
Let the assumption~\ref{ass-g} hold. 
Then,  there exists a function $\varepsilon(R)\to 0$ such that for $y\in K: d(y)>R$ 
    \begin{align}
        \left|U_{\beta} (y)\right|&\le \varepsilon(R) u(y)\label{eq:U.zero}\\
        \left|\frac{\partial U_{\beta}(y)}{\partial y_i}\right|&\le 
        \varepsilon(R)
        \frac{u(y)}{d(y)}+Cd(y)\beta(y)\label{eq:U.first}\\
         \left|\frac{\partial^2 U_{\beta}(y)}{\partial y_i y_j}\right|   
                                                               &\le \varepsilon(R)                                                                 \frac{u(y)}{d(y)^{2}}+C\beta(y)\label{eq:U.second}\\
                                                               [(U_\beta)_{y_iy_j}]_{\theta, B(y,\frac{1}{3}d(y))} 
                                                               &\le \varepsilon(R)\frac{u(y)}{d(y)^{2+\theta}}
        +C \frac{\beta(y)}{d(y)^{\theta}},
        \label{eq:U.third}
    \end{align}
    where $\theta\in (0,1]$. 
\end{lemma}    
\begin{proof}
    The first bound~\eqref{eq:U.zero}  is simply~\eqref{eq:ubeta.bound0}.

    For other bounds we make use of Theorem~4.6 in~\cite{GT2001}. 
        We apply this theorem to the concentric balls $B_r(y)$ and $B_{2r}(y)$, where $r=\frac{1}{3}d(y)$ to obtain that 
    \begin{align*}
        r(U_\beta(y))_{y_i} + r^2 (U_\beta)_{y_iy_j}(y)&\le C(d,\theta)
        \left(\sup_{x:x\in B(y,\frac{2}{3}d(y))} U_\beta(x)+ r^{2+\theta}[\beta]_{\theta, B(y,2r)}\right) \\ 
        r^{2+\theta}[(U_\beta)_{y_iy_j}]_{\theta, B(y,r)}&\le C(d,\theta)\left(\sup_{x:x\in B(y,\frac{2}{3}d(y))} U_\beta(x)+ r^{2+\theta}[\beta]_{\theta , B(y,2r)}   \right). 
    \end{align*}
    It is known that  the distance $d(x)$ to the boundary of $K$ is 
     uniformly Lipschitz, see e.g.~\cite[Section 14.6]{GT2001}, 
    \[
        |d(x)-d(z)|\le |x-z|,\quad  x,z \in K. 
    \]
    We have, for $x,z\in B\left(y,\frac{1}{3}d(y)\right)$,
    \begin{align*}
        \frac{|\beta(x)-\beta(z)|}{|x-z|^\theta}
        &\le 
        \frac{||x|^{p-1}-|z|^{p-1}|}{|x-z|^\theta}
        \frac{\gamma(d(z))}{d(z)}
        +|z|^{p-1} \frac{|\gamma(d(x))/d(x)-\gamma(d(z))/d(z)}{|x-z|^\theta}  \\
        &\le 
        C |y|^{p-2}\frac{\gamma(d(y))}{d(y)} 
        d(y)^{1-\theta}
        +C|y|^{p-1}\frac{\gamma(d(y))}{d(y)^2}d(y)^{1-\theta}\\
        &\le C\frac{\beta (y)}{d(y)^\theta}, 
\end{align*}    
where we used the mean value theorem and the assumption that $\gamma$ is differentiable. 

\end{proof}
 
\begin{lemma}\label{lem:u-beta-diff}
  Assume that~\eqref{eq:harmonic.boundary} holds. 
  For any $\varepsilon>0$ there exists $R>0$ such that for 
    $x\in K$ with  $d(x)>R$
\begin{equation}
    \label{diff-bound1-beta}
    |U_\beta(x+y)-U_\beta(x)|\le  \varepsilon |y||x|^{p-1}, 
  \end{equation}
for $|y|\le |x|/2$.   
 \end{lemma}  
\begin{proof} 
    The proof follows the same arguments as the proof of~\cite[Lemma~2.3]{DW19}
The only change is that we use~\eqref{eq:U.first} to estimate the first derivative.  
\end{proof}
Put 
\[
    f_\beta(x):=\e[U_{\beta} (x+X)] - U_{\beta}(x)
\]
\begin{lemma}\label{lem:supermartingale.drift}
    For any $\varepsilon>0$ there exists $R>0$ 
    such that for 
    $x\in K$ with  $d(x)>R$ the following bound holds 
    \[
        \left|f_\beta(x)+\frac{1}{2}\beta(x)\right| 
        \le \varepsilon \beta(x).
    \]
\end{lemma}
\begin{proof}
    The proof goes exactly as in Lemma~\ref{lem:bound.f}. 
    We will apply~\eqref{eq:taylor.stochastic} with $U=U_\beta$. 
    The remainder $R_{2,\theta}(x)$ can be estimated using the  the inequality~\eqref{eq:U.third}, 
      \[
          R_{2,\theta}(x)=
          C_\alpha
          \left(
              \varepsilon \frac{u(x)}{d(x)^{2+\theta}}+\frac{\beta(x)}{d(x)^{\theta}}
          \right).
\] Also, using Lemma~\ref{lem:ubeta.bound0} and~\eqref{diff-bound1-beta}, 
    we can estimate 
    \begin{align*}
      &  \left|\e \left(U_\beta(x+X)-U_\beta(x)\right) ;|X|>d(x)) \right|\\ 
      &\hspace{1cm}\le 
      C\varepsilon(R)\e[|X|^p;|X|>|x|/2)]
      +C\varepsilon(R)|x|^{p-1}\e[|X|;|X|>d(x)/2],
    \end{align*}
    where $\varepsilon(R)\to 0$.
Then, applying~\eqref{eq:taylor.stochastic} with $U=U_\beta$ and  
using the fact that  $\Delta U_\beta(x)=-\beta(X)$ we obtain, 
\begin{align*}
        \left|f_\beta(x)+\frac{1}{2}\beta(x)\right|
             &\le \left|\e\left[\left(\nabla U_\beta(x)\cdot X +\frac{1}{2}\sum_{i,j}\frac{\partial^2 U_\beta}{\partial x_i\partial x_j}X_iX_j\right){\rm 1}(|X|>d(x)/2)\right]\right|\\
             &\hspace{1cm}+CR_{2,\theta}(x)\e \left[|X|^{2+\theta};|X|\le d (x)/2\right]\\
             &+C\varepsilon(R)\e[|X|^p;|X|>|x|/2]
             +C\varepsilon(R)|x|^{p-1}\e[|X|;|X|>d(x)/2]. 
    \end{align*}
    The partial derivatives of the function $U_\beta$
    in the first term can be  estimated via Lemma~\ref{lem:drift.main.term}, which results in the following estimate 
    \begin{align*}
\left|\e\left[\nabla U_\beta(x)\cdot X ;|X|>d(x)/2\right]\right|\le
C\left(\varepsilon(R)\frac{u(x)}{d(x)}+C\beta(x)d(x)\right)\e[|X|;|X|>d(x)/2]. 
    \end{align*}   
    Using~\eqref{eq:gamma.integral.finite} with $p=1$ we obtain 
    \[
\left|\e\left[\nabla U_\beta(x)\cdot X ;|X|>d(x)/2\right]\right|
\le C  \varepsilon(R)\beta(x).
    \]
We can estimate the terms with the second derivative as follows.     
\begin{align*}
&\frac{1}{2}\left|\e\left[\sum_{i,j}\frac{\partial^2 u(x)}{\partial x_i\partial x_j}X_iX_j;|X|>d(x)/2\right]\right|\\
&\hspace{1cm}\le 
C\left(
    \varepsilon(R)\frac{u(x)}{d(x)^2}+C \beta(x)
\right)
\e[|X|^2;|X|>d(x)/2]. 
\end{align*}    
Then, 
\begin{align*}
 \left|f_\beta(x)+\frac{1}{2}\beta(x)\right|
        &\le C\biggl(
            \varepsilon(R)\beta(x)
            +\left(\varepsilon(R)\frac{u(x)}{d(x)^2}+C \beta(x)\right)\frac{\gamma(d(x))}{d(x)}
          \\
            &\hspace{1cm}
            +R_{2,\theta}(x)\e \left[|X|^{2+\theta};|X|\le d (x)/2\right]\\
            &\hspace{1cm}+\varepsilon(R)\e[|X|^p;|X|>|x|/2]
            +\varepsilon(R)|x|^{p-1}\e[|X|;|X|>d(x)/2]
      \biggr).
    \end{align*}
    Using the fact that  $\varepsilon(R)\to 0$ as $R\to \infty$ and 
 Lemma~\ref{lem:construction.gamma} we arrive at the conclusion.  
\end{proof}    

\begin{lemma}
\label{lem:submart}
For every $c\le 1/2$ there exists $R>0$ such that the sequence
\begin{align}
\label{eq:def.Y}
\nonumber
Y_n^{(c)}
&:=\left(u(x+Rx_0+S(n))+U_\beta(x+Rx_0+S(n))\right)
{\rm 1}\{\tau_x>n\}\\
&\hspace{3cm}
+c\sum_{k=0}^{n-1}\beta(x+Rx_0+S(k)){\rm 1}\{\tau_x>k\}
\end{align}
is a supermartingale.
\end{lemma}
\begin{proof}
Set, for brevity,
$$
V_\beta(y):=u(Rx_0+y)+U_\beta(Rx_0+y).
$$
Then 
\begin{align*}
&\e[Y_n^{(c)}-Y_{n-1}^{(c)}|\mathcal{F}_{n-1}]\\
&\hspace{1cm}=\e[V_\beta(x+S(n)){\rm 1}\{\tau_x>n\}
-V_\beta(x+S(n-1)){\rm 1}\{\tau_x>n-1\}|\mathcal{F}_{n-1}]\\
&\hspace{2cm}+c\beta(x+Rx_0+S(n-1)){\rm 1}\{\tau_x>n-1\}.
\end{align*}
Since the functions $u$ and $U_\beta$ are nonnegative, we have the inequality
\begin{align*}
&\e[Y_n^{(c)}-Y_{n-1}^{(c)}|\mathcal{F}_{n-1}]\\
&\hspace{1cm}\le{\rm 1}\{\tau_x>n-1\}\e[V_\beta(x+S(n))
-V_\beta(x+S(n-1))|\mathcal{F}_{n-1}]\\
&\hspace{2cm}+c\beta(x+Rx_0+S(n-1)){\rm 1}\{\tau_x>n-1\}\\
&\hspace{1cm}=\left(f(x+Rx_0+S(n-1))+f_\beta(x+Rx_0+S(n-1))\right)
{\rm 1}\{\tau_x>n-1\}\\
&\hspace{2cm}+c\beta(x+Rx_0+S(n-1)){\rm 1}\{\tau_x>n-1\}.
\end{align*}
According to Lemma~\ref{lem:bound.f} and Lemma~\ref{lem:supermartingale.drift} there exists $R>0$ such that 
$$
|f(y+Rx_0)|\le\left(\frac{1}{8}-\frac{c}{4}\right)\beta(y+Rx_0)
$$
and
$$
\left|f_\beta(y+Rx_0)+\frac{1}{2}\beta(y+Rx_0)\right|
\le\left(\frac{1}{8}-\frac{c}{4}\right)\beta(y+Rx_0)
$$
for all $y\in K$. Therefore,
$$
\e[Y_n^{(c)}-Y_{n-1}^{(c)}|\mathcal{F}_{n-1}]
\le -\left(\frac{1}{4}-\frac{c}{2}\right)
\beta(x+Rx_0+S(n-1)){\rm 1}\{\tau_x>n-1\}.
$$
This completes the proof.
\end{proof}
\begin{lemma}
\label{lem:beta_bound}
If $R$ is large enough then 
\begin{equation}
\label{eq:beta_1} 
\e\left[\sum_{k=0}^{\tau_x-1}\beta(x+Rx_0+S(k))\right]
\le 3V_\beta(x).
\end{equation}
Furthermore, there exists $\varepsilon(R)\downarrow 0$ such that
\begin{equation}
\label{eq:beta_2} 
\e\left[\sum_{k=0}^{\tau_x-1}|f(x+Rx_0+S(k))|\right]
\le \varepsilon(R)u(x+Rx_0).
\end{equation}
\end{lemma}
\begin{proof}
Let $R$ be so large that $(Y_n^{(1/3)})_{n\ge 0}$ defined in ~\eqref{eq:def.Y} 
is a supermartingale.
Then, 
$$
E[Y_n^{(1/3)}]\le Y_0^{(1/3)}=V_\beta(x). 
$$
Since $u$ and $U_\beta$ are non-negative, 
$$
\frac{1}{3}\e\left[\sum_{k=0}^{n-1}\beta(x+Rx_0+S(k))
{\rm 1}\{\tau_x>k\}\right]
\le V_\beta(x).
$$
Letting here $n\to\infty$ and noting that 
$$
\sum_{k=0}^{\infty}\beta(x+Rx_0+S(k)){\rm 1}\{\tau_x>k\}
=\sum_{k=0}^{\tau_x-1}\beta(x+Rx_0+S(k))
$$
we obtain \eqref{eq:beta_1}. 

Estimate~\eqref{eq:beta_2} follows from \eqref{eq:beta_1}, \eqref{eq:U.zero} 
and from Lemma~\ref{lem:bound.f}. 
\end{proof}
\section{ Construction of the harmonic function     }
Consider the sequence 
$$
Z_n:=u(x+Rx_0+S(n\wedge\tau_x))
-\sum_{k=0}^{n\wedge\tau_x-1}f(x+Rx_0+S(k)),\quad n\ge0.
$$
Then
\begin{align*}
&\e[Z_n-Z_{n-1}|\mathcal{F}_{n-1}]\\
&\hspace{1cm}
=\e[(Z_n-Z_{n-1}){\rm 1}\{\tau_x>n-1\}|\mathcal{F}_{n-1}]\\
&\hspace{1cm}
={\rm 1}\{\tau_x>n-1\}
\e[(u(x+Rx_0+S(n))-u(x+Rx_0+S(n-1)))|\mathcal{F}_{n-1}]\\
&\hspace{3cm}-{\rm 1}\{\tau_x>n-1\}f(x+Rx_0+S(n-1)).
\end{align*}
Recalling that $f(y)=\e[u(y+X)]-u(y)$, we conclude that $Z_n$ is a martingale. 
By the optional stopping theorem applied to this martingale,
\begin{align*}
u(x+Rx_0)
&=\e[Z_0]=\e[Z_n]\\
&=\e\left[u(x+Rx_0+S(n))
-\sum_{k=0}^{n-1}f(x+Rx_0+S(k));\tau_x>n\right]\\
&\hspace{1cm}
+\e\left[u(x+Rx_0+S(\tau_x))
-\sum_{k=0}^{\tau_x-1}f(x+Rx_0+S(k));\tau_x\le n\right].
\end{align*}
Consequently,
\begin{align}
\label{eq:proof.1}
\nonumber
\e[u(x+Rx_0+S(n));\tau_x&>n]\\
\nonumber
=u(x+Rx_0)&-\e[u(x+Rx_0+S(\tau_x));\tau_x\le n]\\
\nonumber
&+\e\left[\sum_{k=0}^{\tau_x-1}f(x+Rx_0+S(k));\tau_x\le n\right]\\
&+\e\left[\sum_{k=0}^{n-1}f(x+Rx_0+S(k));\tau_x>n\right].
\end{align}
By the monotone convergence theorem,
\begin{equation}
\label{eq:proof.2}
\lim_{n\to\infty}\e[u(x+Rx_0+S(\tau_x));\tau_x\le n]
=\e[u(x+Rx_0+S(\tau_x))].
\end{equation}
The estimate \eqref{eq:beta_2} allows one to apply the dominated convergence theorem. As a result we have
\begin{equation}
\label{eq:proof.3}
\lim_{n\to\infty}
\e\left[\sum_{k=0}^{\tau_x-1}f(x+Rx_0+S(k));\tau_x\le n\right]
=\e\left[\sum_{k=0}^{\tau_x-1}f(x+Rx_0+S(k))\right]
\end{equation}
and 
\begin{align}
\label{eq:proof.4}
\nonumber
&\limsup_{n\to\infty}\left|\e\left[\sum_{k=0}^{n-1}f(x+Rx_0+S(k));\tau_x>n\right]\right|\\
&\hspace{2cm}\le \limsup_{n\to\infty}\e\left[\sum_{k=0}^{\tau_x-1}|f(x+Rx_0+S(k))|;\tau_x>n\right]=0.
\end{align}
Combining \eqref{eq:proof.1}---\eqref{eq:proof.4}, we conclude that 
\begin{align}
\label{eq:proof.5}
\nonumber
&\lim_{n\to\infty}\e[u(x+Rx_0+S(n));\tau_x>n]\\
&\hspace{5mm}=u(x+Rx_0)-\e[u(x+Rx_0+S(\tau_x))]
+\e\left[\sum_{k=0}^{\tau_x-1}f(x+Rx_0+S(k))\right].
\end{align}
Recalling that  $u(x)$ is nonnegative and taking into account \eqref{eq:beta_2}, we see that this limit is finite.

We now show that the sequence $\e[u(x+S(n));\tau_x>n]$ has the same limit. According to \eqref{diff-bound},
$$
|u(x+Rx_0+S(n))-u(x+S(n))|\le 
C\left(R|x+Rx_0+S(n)|^{p-1}+R^p\right).
$$
Therefore,
\begin{align}
\label{eq:proof.6}
\nonumber
&\Big|\e[u(x+Rx_0+S(n));\tau_x>n]-\e[u(x+S(n));\tau_x>n]\Big|\\
&\hspace{2cm}\le
CR\e[|x+Rx_0+S(n)|^{p-1};\tau_x>n]+CR^p\pr(\tau_x>n).
\end{align}
We next notice that \eqref{eq:beta_1} implies that 
$$
\varepsilon_n:=\e[\beta(x+Rx_0+S(n));\tau_x>n]\to0.
$$
Then there exists a sequence $a_n\uparrow\infty$ such that 
\begin{align}
\label{eq:proof.7}
\nonumber
&\e[|x+Rx_0+S(n)|^{p-1};d(x+Rx_0+S(n))\le a_n, \tau_x>n]\\
\nonumber
&\hspace{2cm}\le 
\frac{a_n}{\gamma(a_n)}\e[\beta(x+Rx_0+S(n));\tau_x>n]\\
&\hspace{2cm}=\frac{a_n}{\gamma(a_n)}\varepsilon_n\to0.
\end{align}
Moreover, using \eqref{eq:harmonic.boundary.lower} and \eqref{eq:proof.5}, we obtain 
\begin{align}
\label{eq:proof.8}
\nonumber
&\e[|x+Rx_0+S(n)|^{p-1};d(x+Rx_0+S(n))>a_n, \tau_x>n]\\
&\hspace{2cm}\le \frac{C}{a_n}\e[u(x+Rx_0+S(n));\tau_x>n]\to0.
\end{align}
Combining \eqref{eq:proof.6}---\eqref{eq:proof.8}
and using the finiteness of $\tau_x$, we conclude that 
$$
\Big|\e[u(x+Rx_0+S(n));\tau_x>n]-\e[u(x+S(n));\tau_x>n]\Big|\to0.
$$
From this and from \eqref{eq:proof.5} we get \eqref{eq:def.V}
with
\begin{equation}
\label{eq:defV}
V(x)=u(x+Rx_0)-\e[u(x+Rx_0+S(\tau_x))]
+\e\left[\sum_{k=0}^{\tau_x-1}f(x+Rx_0+S(k))\right].
\end{equation}
Due to \eqref{eq:beta_2}, 
$$
V(x)\le Cu(x+Rx_0)
$$
This upper bound allows one to repeat the arguments from Lemma~13
in \cite{DW15} and to conclude that $V$ is harmonic for $S(n)$ killed at leaving the cone $K$.
This completes the proof of the first statement  of Theorem~\ref{thm:C2}. 

\section{Upper bounds  for $\pr(\tau_x>n)$ in the case $p\ge 1$} 

We start with the following estimate for tails distribution of exit time $\tau_x^{bm}$ from $K$ of  Brownian motion. 
\begin{lemma}
    Assume that $p\ge 1$ and that  the cone $K$ is Lipschitz and starlike.
    Then there exists a constant $C$ such that 
    \begin{equation}\label{eq:bound.tail.tau}
        \pr(\tau_x^{bm }>t) \le C \frac{u(x)}{t^{p/2}},\quad x\in K. 
    \end{equation}
\end{lemma}    
\begin{proof} 
    For convex cones the proof can be found in (0.4.1) of \cite{Var99}
    We will show how the proof of part (ii) in Theorem~1 of~\cite{Var99} at page 344
    can be modified to show the result under our assumptions. 

    First note that for $x\in x_0+K$ we have $d(x)\ge c_0$  for some $c_0$. 
    Then, for  $x\in x_0+K\subset K$, 
    \[
    \pr(\tau_x^{bm}>1)  \le 1\le \frac{1}{c_0}  d(x) \le 
    C d(x)|x|^{p-1}\le C u(x),
    \]
    where we also used ~\eqref{eq:harmonic.boundary.lower} and the assumption 
    $p\ge 1$. 
    Now note that that the rest of the proof in~~\cite{Var99} does not use convexity and uses 
    only the assumption that $K$ is Lipschitz. 
    Thus, repeating the rest of the proof of part (ii) in Theorem~1 of~\cite{Var99} at page 344 in exactly 
    the same way we obtain that  
    \[
        \pr(\tau_x^{bm}>1) \le Cu(x).
    \]
    Then, the statement follows by the scaling property of Brownian motion. 
\end{proof}    

\begin{lemma}
\label{lem:ub1}
Assume that the conditions of Theorem~\ref{thm:C2} are valid.
Assume that $p\ge 1$.  
Then,  
$$
\pr(\tau_x>n)\le C\frac{V_\beta(x)}{n^{p/2}}
\le C\frac{u(x+Rx_0)}{n^{p/2}}.
$$
\end{lemma}
\begin{proof}
Choose $b_n=o(\sqrt{n})$ so that
\begin{equation}
\label{eq:def.b_n}
\pr\left(\sup_{u\le n}|S([u])-B(u)|\ge b_n\right)
=o\left(\left(\frac{b_n}{\sqrt{n}}\right)^p\right).
\end{equation}
Using this approximation, we get
$$
\pr(\tau_y>k)\le \pr(\tau_{y+b_nx_0}^{bm}>k)
+o\left(\left(\frac{b_n}{\sqrt{n}}\right)^p\right),\quad k\le n.
$$
Then, using~\eqref{eq:bound.tail.tau}, 
$$
\pr(\tau_y>k)\le C\frac{u(y+b_nx_0)}{k^{p/2}}
+o\left(\left(\frac{b_n}{\sqrt{n}}\right)^p\right),\quad k\le n.
$$
If $d(y)\ge b_n$ then, using \eqref{diff-bound} and \eqref{eq:harmonic.boundary.lower}, we obtain
$$
u(y+b_nx_0)\le u(y)+Cb_n(|y|^{p-1}+b_n^{p-1})\le C u(y). 
$$
As a result, for all $y\in K$ with $d(y)\ge b_n$,
$$
\pr(\tau_y>n/4)\le C\frac{u(y)}{n^{p/2}}.
$$
Define
$$
\nu_n:=\inf\{k\ge1: d(x+S(k))\ge b_n\}.
$$
Then, by the Markov property,
\begin{align}
\label{upp.bound.0}
\nonumber
\pr(\tau_x>n,\nu_n\le 3n/4)
&\le C\frac{\e[u(x+S(\nu_n));\tau_x>\nu_n,\nu_n\le 3n/4]}{n^{p/2}}
\\
\nonumber
&\le C\frac{\e[u(x+Rx_0+S(\nu_n));\tau_x>\nu_n,\nu_n\le 3n/4]}{n^{p/2}}\\
\nonumber
&\le C\frac{\e[V_\beta(x+S(\nu_n));\tau_x>\nu_n,\nu_n\le 3n/4]}{n^{p/2}}\\
&\le C\frac{\e[Y^{(0)}_{\nu_n};\nu_n\le 3n/4]}{n^{p/2}},
\end{align}
where $Y^{(0)}_n$ is defined as in Lemma~\ref{lem:submart}.  
Since this sequence is a supermartingale,
$$
\e[Y^{(0)}_{\nu_n};\nu_n\le 3n/4]
\le \e[Y^{(0)}_{\nu_n\wedge(3n/4)}]
\le Y_0^{(0)}=V_\beta(x).
$$
Therefore,
\begin{equation}
\label{upp.bound.1}
\pr(\tau_x>n,\nu_n\le 3n/4)\le C\frac{V_\beta(x)}{n^{p/2}}.
\end{equation}

Set 
$$
K_n:=\{y\in K:\, d(y)\ge b_n\}
$$
and 
$$
c_n:=[b_n^2].
$$
Then, by the Markov property,
\begin{align*}
&\pr(\tau_x>n,\nu_n>3n/4)\\
&\hspace{1cm}
\le\pr\left(\tau_x>n/2,x+S(n/2+jc_n)\in K\setminus K_n,
\text{ for all } j\le n/4c_n\right)\\
&\hspace{1cm}
\le\pr(\tau_x>n/2)\left(\sup_{y\in K\setminus K_n}
\pr\left(y+S(c_n)\in K\setminus K_n\right)\right)^{n/4c_n}. 
\end{align*}
Repeating now the arguments from the proof of Lemma 14 in \cite{DW15}, we obtain
$$
\pr(\tau_x>n,\nu_n>3n/4)\le \pr(\tau_x>n/2)e^{-cn/b_n^2}.
$$
Combining this with \eqref{upp.bound.1}, we have
$$
\pr(\tau_x>n)\le C\frac{V_\beta(x)}{n^{p/2}}
+\pr(\tau_x>n/2)e^{-cn/b_n^2}.
$$
Set $\gamma_j:=c2^j/b^2_{2^j}$. Then, letting $n=2^N$ in the previous bound, we have 
$$
\pr(\tau_x>2^N)\le C\frac{V_\beta(x)}{2^{Np/2}}
+e^{-\gamma_N}\pr(\tau_x>2^{N-1}).
$$
Iterating this estimate $m$ times, we obtain
\begin{align*}
\pr(\tau_x>2^N)
&\le C\frac{V_\beta(x)}{2^{Np/2}}
\left(1+\sum_{k=0}^{m-2}2^{(k+1)p/2}
\prod_{j=0}^ke^{-\gamma_{N-j}}\right)\\
&\hspace{2cm}+\prod_{j=0}^{m-1}e^{-\gamma_{N-j}}
\pr(\tau_x>2^{N-m}). 
\end{align*}
Choosing here $m=[N/2]$ and noting that
$\min_{j\le N/2}\gamma_{N-j}$ converges to infinity as
$N\to\infty$, we infer that 
$$
\prod_{j=0}^{m-1}e^{-\gamma_{N-j}}=o(2^{-Np/2})
$$
and
$$
\sup_{N\ge1}
\sum_{k=0}^{m-2}2^{(k+1)p/2}\prod_{j=0}^ke^{-\gamma_{N-j}}
<\infty.
$$
This finishes the proof.
\end{proof}
We can now give a proof  of a strengthened version of 
one of results in McConnell~\cite{MC84}. 
\begin{corollary}
For every $0<q<p$,
\begin{equation}
\label{eq:maximum}
\e\left[\max_{k\le\tau_x}|S(k)|^q\right]
\le C(V_\beta(x))^{q/p}.
\end{equation}
\end{corollary}
\begin{proof}
For every $q<p$, by the Burkholder-Davis-Gundy inequality,
$$
\e\left[\max_{k\le\tau_x}|S(k)|^q\right]
\le C\e[(\tau_x)^{q/2}].
$$
Using Lemma~\ref{lem:ub1}, we have 
\begin{align*}
\e[(\tau_x)^{q/2}]
&=\frac{q}{2}\int_0^\infty u^{q/2-1}\pr(\tau_x>u)du\\
&\le K^{q/2}+\frac{q}{2}\int_K^\infty u^{q/2-1}\pr(\tau_x>u)du\\
&\le K^{q/2}+CV_\beta(x)\frac{q}{2}\int_K^\infty u^{q/2-p/2-1}du\\
&=K^{q/2}+CV_\beta(x)\frac{q}{p-q}K^{(q-p)/2}.
\end{align*}
Taking here $K=(V_\beta(x))^{2/p}$, we have 
$$
\e[(\tau_x)^{q/2}]\le C(V_\beta(x))^{q/p}.
$$
\end{proof}

\begin{lemma}
\label{lem:ub2} If $u(x+S(n)){\rm 1}\{\tau_x>n\}$ is a submartingale then
$$
\pr(\tau_x>n)\le C\frac{\e[u(x+S(n));\tau_x>n]}{n^{p/2}}.
$$
\end{lemma}
\begin{proof}
By the first inequality in \eqref{upp.bound.0},
\begin{align*}
\pr(\tau_x>n,\nu_n\le3n/4)
&\le c\frac{\e[u(x+S(\nu_n));\tau_x>\nu_n,\nu_n\le 3n/4]}{n^{p/2}}
\\
&\le c\frac{\e[u(x+S(\nu_n\wedge(3n/4)));\tau_x>\nu_n\wedge 3n/4]}{n^{p/2}}.
\end{align*}
Due to the submartingale property of
$u(x+S(n)){\rm 1}\{\tau_x>n\}$,
$$
\e[u(x+S(\nu_n\wedge(3n/4)));\tau_x>\nu_n\wedge 3n/4]
\le \e[u(x+S(n));\tau_x>n].
$$
Consequently,
$$
\pr(\tau_x>n,\nu_n\le3n/4)
\le C\frac{\e[u(x+S(n));\tau_x>n]}{n^{p/2}}.
$$
Repeating now the second part of the proof of the previous lemma 
and using the monotonicity of $\e[u(x+S(n));\tau_x>n]$, we get the desired result.
\end{proof}
\section{Asymptotic properties of  the harmonic function.}

In order to show the equivalence $V(x)\sim u(x)$ we first notice that the assumption $d(x)\to\infty$ implies that 
$u(x)\sim u(x+Rx_0)$ for every fixed $R$ 
when $u(x)\to+\infty$. 
Next, by \eqref{eq:beta_2},
\begin{align}
\label{eq:proof.9}
\nonumber
\left|\frac{V(x)}{u(x+Rx_0)}-1\right|
&\le \frac{\e[u(x+Rx_0+S(\tau_x))]}{u(x+Rx_0)}
+\frac{\left|\e\left[\sum_{k=0}^{\tau_x-1}f(x+Rx_0+S(k))\right]\right|}{u(x+Rx_0)}\\
&\le \frac{\e[u(x+Rx_0+S(\tau_x))]}{u(x+Rx_0)}
+\varepsilon(R).
\end{align}
When $p\ge 1$ and $x+S(\tau_x)\notin K$, we obtain from~\eqref{diff-bound} 
that 
\begin{equation}
\label{eq:proof.9.1}
u(x+Rx_0+S(\tau_x))
\le CR|x+Rx_0+S(\tau_x)|^{p-1}
\le CR^p+C|x|^{p-1}+C|S(\tau_x)|^{p-1}.
\end{equation}
It follows from~\eqref{eq:maximum} that 
\begin{equation}
\label{eq:proof.10}
\e[|S(\tau_x)|^{q}]\le C(1+|x|^{q}),\quad q<p.
\end{equation}
Consequently,
$$
\e[u(x+Rx_0+S(\tau_x))]\le CR^p+C|x|^{p-1}.
$$
Applying now \eqref{eq:harmonic.boundary.lower}, we get 
$$
\e[u(x+Rx_0+S(\tau_x))]\le CR^p+\frac{C}{d(x)}u(x)
\le CR^p+\frac{C}{d(x)}u(x+Rx_0)
$$
and, consequently,
\begin{equation}
    \label{eq:proof.10.1}
\frac{\e[u(x+Rx_0+S(\tau_x))]}{u(x+Rx_0)}\to0.
\end{equation} 
When $p<1$ the random variable $u(x+Rx_0+S(\tau_x))$ is bounded. 
Hence~\eqref{eq:proof.10.1} holds when $u(x)\to +\infty,$ or, equivalently, 
when $d(x)|x|^{p-1}\to+\infty$.

Combining  convergence  in~ \eqref{eq:proof.10.1} 
with \eqref{eq:proof.9} and using the relation
$u(x)\sim u(x+Rx_0)$, we conclude that 
$$
\limsup \left|\frac{V(x)}{u(x)}-1\right|
\le \varepsilon(R). 
$$
Letting here $R\to\infty$ we complete the proof.

We now turn to the proof of the uniform convergence 
in the case $p\ge 1$. 
We first prove the following auxiliary result:
\begin{align}
\label{eq:maximum1}
\nonumber
&\mathbf{E}[|S(n)|^{p-1};\tau_x>n]
+\mathbf{E}[|S(\tau_x)|^{p-1};\tau_x>n]\\
&\hspace{1cm}\le2\mathbf{E}[\max_{k\le\tau_x}|S(k)|^{p-1};\tau_x>n]=o(V_\beta(x))
\end{align}
uniformly in $x\in K$. Fix some $p'>p$. Then 
$\frac{1}{q'}:=1-\frac{1}{p'}>1-\frac{1}{p}$ and, consequently,
$q'<p/(p-1)$. Using now the H\"older inequality,  \eqref{eq:maximum} with $q=q'(p-1)$ 
and Lemma~\ref{lem:ub1}, we get 
\begin{align*}
\mathbf{E}[\max_{k\le\tau_x}|S(k)|^{p-1};\tau_x>n]
&\le\left(\mathbf{E}[\max_{k\le\tau_x}|S(k)|^{q'(p-1)}]\right)^{1/q'}\pr^{1/p'}(\tau_x>n)\\
&\le CV_\beta^{1-1/p}(x)\frac{V_\beta^{1/p'}(x)}{n^{p/2p'}}.
\end{align*}
Since $V_\beta^{1/p'-1/p}(x)$ is bounded, we get \eqref{eq:maximum1}.

Combining \eqref{eq:proof.6}, \eqref{eq:maximum1} and Lemma~\ref{lem:ub1}, we have
\begin{align*}
&\Big|\e[u(x+Rx_0+S(n));\tau_x>n]-\e[u(x+S(n));\tau_x>n]\Big|\\
&\hspace{2cm}\le
C(R^p+R|x|^{p-1})\pr(\tau_x>n)+
\e[|S(n)|^{p-1};\tau_x>n]\\
&\hspace{2cm}\le C\frac{R^p+R|x|^{p-1}}{n^{p/2}}V_\beta(x)
+o(V_\beta(x)).
\end{align*}
Therefore,
$$
\Big|\e[u(x+Rx_0+S(n));\tau_x>n]-\e[u(x+S(n));\tau_x>n]\Big|
=o(V_\beta(x))
$$
uniformly in $x\in K$ such that $|x|=o(n^{p/2(p-1)})$. 
This implies that we are left to show that
\begin{equation}
\label{eq:uniform}
\e[u(x+Rx_0+S(n));\tau_x>n]-V(x)=o(V_\beta(x))
\end{equation}
uniformly in $x\in K$ such that $|x|=o(n^{p/2(p-1)})$.
Combining \eqref{eq:defV} and \eqref{eq:proof.1}, we have 
\begin{align*}
&\e[u(x+Rx_0+S(n));\tau_x>n]-V(x)\\
&\hspace{1cm}=\e[u(x+Rx_0+S(\tau_x));\tau_x>n]
-\e\left[\sum_{k=n}^{\tau_x-1}f(x+Rx_0+S(k));\tau_x>n\right].
\end{align*}
Using~\eqref{eq:proof.9.1} 
 and~\eqref{eq:maximum1}
we obtain 
\begin{multline*}
    \e[u(x+Rx_0+S(\tau_x));\tau_x>n]\\ 
    \le 
    C(R^p+R|x|^{p-1})\pr(\tau_x>n)+
\e[|S(\tau_x)|^{p-1};\tau_x>n]
=o(V_\beta(x)), 
\end{multline*}    
for $x\in K$ such that $|x|=o(n^{p/2(p-1)})$. 
For $y\in K$ now put  
\[
Q(y):=     \sum_{k=0}^{\tau_y-1}f(y+Rx_0+S(k)). 
\]
and note that,  by the Markov property, 
\[
    \e\left[\sum_{k=n}^{\tau_x-1}f(x+Rx_0+S(k));\tau_x>n\right]
    = \e[Q(x+S(n));\tau_x>n]. 
\]
By Theorem~\ref{thm:C2}, as $d(x)\to \infty$, $V(x)\sim u(x)$ and, therefore, 
$Q(x)=o(u(x))$ as $d(x)\to \infty$. 
Consider an increasing sequence $\gamma_n\uparrow \infty$
and  split 
\begin{multline*}
    \e[Q(x+S(n));\tau_x>n]\\
    =
    \e[Q(x+S(n));\tau_x>n, d(x+S(n))\ge \gamma_n]\\ 
    +\e[Q(x+S(n));\tau_x>n, d(x+S(n))< \gamma_n]. 
\end{multline*}
Then,  for some $\varepsilon_n\downarrow 0$, 
\begin{align*}
    \e[Q(x+S(n));\tau_x>n, d(x+S(n))\ge \gamma_n]
    &\le \varepsilon_n 
    \e[u(x+S(n));\tau_x>n]\\ 
    \le C \varepsilon_n  
    \e[V_\beta(x+S(n));\tau_x>n] 
    &\le C \varepsilon_n   V_\beta(x) = o(V_\beta(x)) 
\end{align*}
uniformly in $x\in K$, where we have also used the supermartingale 
property of  $V_\beta(x)$. 
By~\eqref{eq:beta_1}, 
\begin{multline*}
    \e[Q(x+S(n));\tau_x>n, d(x+S(n))< \gamma_n]\\ 
    \le 3 
    \e[V_\beta(x+S(n));\tau_x>n, d(x+S(n))< \gamma_n]\\
    \le 
    C\e[u(x+Rx_0+S(n));\tau_x>n, d(x+S(n))< \gamma_n] \\
    \le C\gamma_n 
    \e[|x+Rx_0+S(n)|^{p-1};\tau_x>n].
\end{multline*}
Then,  for  $\gamma_n \uparrow \infty$ sufficiently slowly, 
by Lemma~\ref{lem:ub1} and~\eqref{eq:maximum1}, 
\begin{align*}
    &\e[Q(x+S(n));\tau_x>n, d(x+S(n))< \gamma_n]\\ 
    &\hspace{1cm}\le 
    C\gamma_n
    \left( |x+Rx_0|^{p-1}
    \pr(\tau_x>n)
    + \e[|x+Rx_0+S(n)|^{p-1};\tau_x>n]
    \right)\\ 
    &\hspace{1cm}= o(V_\beta(x)), 
\end{align*}
uniformly in  $x\in K$ such that $|x|=o(n^{p/2(p-1)})$.  
Thus, the proof of Theorem~\ref{thm:C2} is complete.

\section{proof of Theorem~\ref{thm:smooth_and_convex} for  cones with $1\le p<2$. 
}
Part (a) is immediate from Lemma~\ref{lem:ub1}.

Choose $m=m(n)$ so that $m(n)=o(n)$ and 
$b_n^2=o(m(n))$, where $b_n$ is defined by \eqref{eq:def.b_n}.
Fix some $\varepsilon<1$, $A>1$ and define
\begin{align*}
&K_1:=\{y\in K:\, d(y)\le\varepsilon\sqrt{m}, |y|\le A\sqrt{m}\},\\ 
&K_2:=\{y\in K:\, d(y)>\varepsilon\sqrt{m}, |y|\le A\sqrt{m}\},\\ 
&K_3:=\{y\in K:\, |y|> A\sqrt{m}\}.
\end{align*}
Then, by the Markov property at time $m$,
\begin{align}
\label{eq:proof.11}
\nonumber
\pr&\left(\frac{x+S(n)}{\sqrt{n}}\in D,\tau_x>n\right)\\
\nonumber
=&\int_{K_1}\pr(x+S(m)\in dy,\tau_x>m)
\pr\left(\frac{y+S(n-m)}{\sqrt{n}}\in D,\tau_y>n-m\right)\\
\nonumber
&+\int_{K_2}\pr(x+S(m)\in dy,\tau_x>m)
\pr\left(\frac{y+S(n-m)}{\sqrt{n}}\in D,\tau_y>n-m\right)\\
&+\int_{K_3}\pr(x+S(m)\in dy,\tau_x>m)
\pr\left(\frac{y+S(n-m)}{\sqrt{n}}\in D,\tau_y>n-m\right).
\end{align}
Using the upper bound from (a), obtain we
\begin{align*}
\int_{K_1}&\pr(x+S(m)\in dy,\tau_x>m)
\pr\left(\frac{y+S(n-m)}{\sqrt{n}}\in D,\tau_y>n-m\right)\\
&\le\int_{K_1}\pr(x+S(m)\in dy,\tau_x>m)
\pr\left(\tau_y>n-m\right)\\
&\le\frac{C}{(n-m)^{p/2}}
\int_{K_1}\pr(x+S(m)\in dy,\tau_x>m)u(y+Rx_0)\\
&\le\frac{C}{n^{p/2}}
\int_{K_1}\pr(x+S(m)\in dy,\tau_x>m)u(y+Rx_0).
\end{align*}
It follows from \eqref{eq:harmonic.boundary} that 
$$
u(y)\le u(y+Rx_0)\le C|y+Rx_0|^{p-1}d(y+Rx_0)\le C\varepsilon A^{p-1} m^{p/2},
\quad y\in K_1.
$$
Therefore,
\begin{align}
\label{eq:proof.12}
\nonumber
\int_{K_1}&\pr(x+S(m)\in dy,\tau_x>m)
\pr\left(\frac{y+S(n-m)}{\sqrt{n}}\in D,\tau_y>n-m\right)\\
&\hspace{1cm}
\le C\varepsilon A^{p-1}\frac{m^{p/2}}{n^{p/2}}\pr(\tau_x>m)
\le C\varepsilon A^{p-1}\frac{u(x+Rx_0)}{n^{p/2}}
\end{align}
and 
\begin{align}
\label{eq:proof.13}
\int_{K_1}\pr(x+S(m)\in dy,\tau_x>m)u(y)\le C\varepsilon A^{p-1}u(x+Rx_0).
\end{align}

Using the upper bound from (a) once again, we have 
\begin{align*}
\int_{K_3}&\pr(x+S(m)\in dy,\tau_x>m)
\pr\left(\frac{y+S(n-m)}{\sqrt{n}}\in D,\tau_y>n-m\right)\\
&\le \frac{C}{n^{p/2}}
\int_{K_3}\pr(x+S(m)\in dy,\tau_x>m)u(y+Rx_0)\\
&\le \frac{C}{n^{p/2}}\e\left[|S(m)|^p;\tau_x>m,
|S(m)|>A\sqrt{m}/2\right],
\end{align*}
uniformly in $|x\le \sqrt{m}/2$. 
Applying now the Markov inequality, we obtain 
\begin{align*}
\int_{K_3}&\pr(x+S(m)\in dy,\tau_x>m)
\pr\left(\frac{y+S(n-m)}{\sqrt{n}}\in D,\tau_y>n-m\right)\\
&\le \frac{C}{n^{p/2}}(A\sqrt{m})^{p-2}
\e\left[|S(\tau_x\wedge m)|^2\right].
\end{align*}
Applying the Optional Stopping Theorem to the martingale 
$|S(n)|^2-dn $ 
and the upper bound in Lemma~\ref{lem:ub1} 
for the tail of $\tau_x$, we conclude that 
$$
\e\left[|S(\tau_x\wedge m)|^2\right]
\le C\e[\tau_x\wedge m]\le Cu(x+Rx_0)m^{1-p/2}.
$$
As a result we have 
\begin{align}
\label{eq:proof.14}
\nonumber
 \int_{K_3}&\pr(x+S(m)\in dy,\tau_x>m)
\pr\left(\frac{y+S(n-m)}{\sqrt{n}}\in D,\tau_y>n-m\right)\\
&\hspace{1cm}\le \frac{C}{n^{p/2}}A^{p-2}u(x+Rx_0).
\end{align}
By the same arguments, 
\begin{align}
\label{eq:proof.15}
 \int_{K_3}&\pr(x+S(m)\in dy,\tau_x>m)u(y)
 \le \frac{C}{n^{p/2}}A^{p-2}u(x+Rx_0).
\end{align}
Finally, \eqref{eq:def.b_n} implies that 
\begin{align*}
\pr\left(\frac{y+S(n-m)}{\sqrt{n}}\in D,\tau_y>n-m\right)
&\sim \pr\left(\frac{y+B(n-m)}{\sqrt{n}}\in D,\tau^{bm}_y>n-m\right)\\
&\sim\varkappa\frac{u(y)}{n^{p/2}}\int_D u(z)e^{-|z|^2/2}dz
\end{align*}
uniformly in $y\in K_2.$
Consequently,
\begin{align*}
 \int_{K_2}&\pr(x+S(m)\in dy,\tau_x>m)
\pr\left(\frac{y+S(n-m)}{\sqrt{n}}\in D,\tau_y>n-m\right)\\
&\hspace{1cm}\sim
\frac{\varkappa}{n^{p/2}}\int_D u(z)e^{-|z|^2/2}dz
\int_{K_2}\pr(x+S(m)\in dy,\tau_x>m)u(y).
\end{align*}
Combining this with \eqref{eq:proof.13} and \eqref{eq:proof.15},
we get
\begin{align}
\label{eq:proof.16} 
\nonumber
&\Bigg|\int_{K_2}\pr(x+S(m)\in dy,\tau_x>m)
\pr\left(\frac{y+S(n-m)}{\sqrt{n}}\in D,\tau_y>n-m\right)\\
\nonumber
&\hspace{2cm}-\frac{\varkappa}{n^{p/2}}\int_D u(z)e^{-|z|^2/2}dz
\e[u(x+S(m));\tau_x>m]\Bigg|\\
&\hspace{1cm}\le Cu(x+Rx_0)\frac{\varepsilon A^{p-1}+A^{p-2}}{n^{p/2}}
+o\left(\frac{1}{n^{p/2}}\right).
\end{align}
Plugging \eqref{eq:proof.12}, \eqref{eq:proof.14} and \eqref{eq:proof.16} into \eqref{eq:proof.11}, we obtain 
\begin{align*}
&\Bigg|\pr\left(\frac{x+S(n)}{\sqrt{n}}\in D,\tau_x>n\right)
-\frac{\varkappa}{n^{p/2}}\int_D u(z)e^{-|z|^2/2}dz
\e[u(x+S(m));\tau_x>m]\Bigg|\\
&\hspace{1cm}\le Cu(x+Rx_0)\frac{\varepsilon A^{p-1}+A^{p-2}}{n^{p/2}}
+o\left(\frac{1}{n^{p/2}}\right), 
\end{align*}
uniformly in $|x|\le \sqrt{m}$. 
Letting here first $\varepsilon\to0$ and then $A\to\infty$
and recalling that $\e[u(x+S(m));\tau_x>m]\to V(x)$ 
uniformly in $|x|\le \sqrt{m}$, we finally arrive at the relation 
$$
\pr\left(\frac{x+S(n)}{\sqrt{n}}\in D,\tau_x>n\right)
\sim\frac{\varkappa V(x)}{n^{p/2}}\int_D u(z)e^{-|z|^2/2}dz.
$$
Even simpler arguments give
$$
\pr\left(\tau_x>n\right)
\sim\frac{\varkappa V(x)}{n^{p/2}}.
$$
Thus, (b)  is  proven for $1\le p<2$. 
\section{proof of Theorem~\ref{thm:smooth_and_convex} for $p\ge2$.}
According to Corollary 3.2 in Sakhanenko~\cite{Sakh}, one can construct $S(n)$ and a Brownian motion $B(t)$ on a joint probability space so that
\begin{multline*}
    \pr\left(\sup_{u\le n}|S([u])-B(u)|>x\right)
    \\ 
    \le Cn\left(\e\min\left\{\frac{|X|^2}{x^2},\frac{|X|^3}{x^3}\right\}
    +\e\min\left\{\frac{|B(1)|^2}{x^2},\frac{|B(1)|^3}{x^3}\right\}
    \right).
\end{multline*}
We next notice that the assumption $\e[|X|^2\log|X|]<\infty$ implies that
$$
\e\left[|X|^2;|X|>x\right]\le\frac{1}{\log x}\e\left[|X|^2\log|X|;|X|>x\right]=o\left(\frac{1}{\log x}\right)
$$
and 
\begin{align*}
\e\left[|X|^3;|X|\le x\right]
&\le x^{3/4}+\e\left[|X|^3;x^{1/4}<|X|\le x\right]\\
&\le x^{3/4}+\frac{x}{\log x}\e\left[|X|^2\log |X|;x^{1/4}<|X|\le x\right]=o\left(\frac{x}{\log x}\right).
\end{align*}
Consequently,
\begin{multline*}
    \e\min\left\{\frac{|X|^2}{x^2},\frac{|X|^3}{x^3}\right\}\\ 
    =\frac{1}{x^3}\e\left[|X|^3;|X|\le x\right]
    +\frac{1}{x^2}\e\left[|X|^3;|X|>x\right]
    =o\left(\frac{1}{x^2\log x}\right).
\end{multline*}
Obviously, the same bound holds for 
$X$ replaced by $B(1)$.  As a result,
$$
\pr\left(\sup_{u\le n}|S([u])-B(u)|>x\right)
=o\left(\frac{n}{x^2\log x}\right).
$$
Then we can choose $x=b_n=o\left(\frac{\sqrt{n}}{\log^{1/4}n}\right)$such that
\begin{equation}
\label{eq:p=2.coupl}
\pr\left(\sup_{u\le n}|S([u])-B(u)|>x\right)
=o\left(\frac{1}{\log^{1/2} n}\right)
=o\left(\frac{b_n^2}{n}\right).
\end{equation}
\begin{lemma}
\label{lem:coupling}
Set $\delta_n=\log^{-p/8}n$.
Uniformly in $y\in K$ such that $d(y)\ge \delta_n\sqrt{n}$ and 
$|y|\le\varepsilon_n\sqrt{n}$ for some $\varepsilon_n\downarrow0$,
$$
\pr\left(\frac{y+S(n)}{\sqrt{n}}\in D,\tau_y>n\right)=(\varkappa+o(1))
\left(\int_D u(z)e^{-|z|^2/2}dz\right)\frac{u(y)}{n^{p/2}},
$$
where $D$ is either a compact subset of $K$ or $D=K$.
\end{lemma}
\begin{proof}
Since the proof is very similar to the beginning of the proof of Lemma~\ref{lem:ub1}, we omit it. In the case $p>2$, a stronger result has been proven in Lemma 20 of \cite{DW15}.  
\end{proof}

Define
\begin{align*}
&K_1:=\{y\in K:\, d(y)\le\delta_n\sqrt{n}, |y|\le \varepsilon_n\sqrt{n}\},\\ 
&K_2:=\{y\in K:\, d(y)>\delta_n\sqrt{n}, |y|\le \varepsilon_n\sqrt{n}\},\\ 
&K_3:=\{y\in K:\, |y|> \varepsilon_n\sqrt{n}\}.
\end{align*}
Set also
$$
m=m_n=\left[\frac{n}{\log^{1/4}n}\right].
$$
Then, by the Markov property at time $m$,
\begin{align}
\label{eq:proof.11.new}
\nonumber
\pr&\left(\frac{x+S(n)}{\sqrt{n}}\in D,\tau_x>n\right)\\
\nonumber
&=\int_{K_2}\pr(x+S(m)\in dy,\tau_x>m)
\pr\left(\frac{y+S(n-m)}{\sqrt{n}}\in D,\tau_y>n-m\right)\\
&+\int_{K_1\cup K_3}\pr(x+S(m)\in dy,\tau_x>m)
\pr\left(\frac{y+S(n-m)}{\sqrt{n}}\in D,\tau_y>n-m\right).
\end{align}
It is immediate from Lemma~\ref{lem:coupling} that 
\begin{align*}
 \int_{K_2}&\pr(x+S(m)\in dy,\tau_x>m)
\pr\left(\frac{y+S(n-m)}{\sqrt{n}}\in D,\tau_y>n-m\right)\\
&\hspace{1cm}=
\frac{\varkappa+o(1)}{n^{p/2}}\int_D u(z)e^{-|z|^2/2}dz
\int_{K_2}\pr(x+S(m)\in dy,\tau_x>m)u(y)\\
&\hspace{1cm}=
\frac{\varkappa+o(1)}{n^{p/2}}\int_D u(z)e^{-|z|^2/2}dz
\Bigl(\e[u(x+S(m));\tau_x>m]\\
&\hspace{4cm}-\e[u(x+S(m));\tau_x>m,
x+S(m)\in K_1\cup K_3]\Bigr).
\end{align*}
Using the upper bound from (a), we obtain
\begin{align*}
\int_{K_1\cup K_3}&\pr(x+S(m)\in dy,\tau_x>m)
\pr\left(\frac{y+S(n-m)}{\sqrt{n}}\in D,\tau_y>n-m\right)\\
&\le\int_{K_1\cup K_3}\pr(x+S(m)\in dy,\tau_x>m)
\pr\left(\tau_y>n-m\right)\\
&\le\frac{C}{(n-m)^{p/2}}
\int_{K_1\cup K_3}\pr(x+S(m)\in dy,\tau_x>m)u(y+Rx_0)\\
&\le\frac{C}{n^{p/2}}
\e[u(x+Rx_0+S(m));\tau_x>m,
x+S(m)\in K_1\cup K_3].
\end{align*}
If we show that
\begin{equation}
\label{eq:K1_K3}
\e[u(x+Rx_0+S(m));\tau_x>m,
x+S(m)\in K_1\cup K_3]=o(n^{p/2}u(x+Rx_0))
\end{equation}
uniformly in $x\in K$ with $|x|=o\left(\frac{\sqrt{n}}{\log n}\right)$ 
then
$$
\pr\left(\frac{x+S(n)}{\sqrt{n}}\in D,\tau_x>n\right)
=\frac{\varkappa+o(1)}{n^{p/2}}\int_D u(z)e^{-|z|^2/2}dz
\e[u(x+S(m));\tau_x>m].
$$
Combining this with Theorem~\ref{thm:C2}, we get the desired relation.
Thus, we are left to prove \eqref{eq:K1_K3}.

It follows from \eqref{eq:harmonic.boundary} that 
$$
u(y+Rx_0)\le C|y+Rx_0|^{p-1}d(y+Rx_0)
\le C \varepsilon_n^{p-1}\delta_nn^{p/2}
\quad y\in K_1.
$$
Therefore,
$$
\e[u(x+Rx_0+S(m));\tau_x>m,x+S(m)\in K_1]
\le C \varepsilon_n^{p-1}\delta_nn^{p/2}\pr(\tau_x>m).
$$
Applying the upper bound from part (a) and recalling the definition of $m$, we conclude that 
\begin{align}
\label{eq:K1}
\nonumber
&\e[u(x+Rx_0+S(m));\tau_x>m,x+S(m)\in K_1]\\
&\hspace{1cm}\le C \varepsilon_n^{p-1}\delta_n(n/m)^{p/2}u(x+Rx_0)=
C \varepsilon_n^{p-1}u(x+Rx_0).
\end{align}

Now let 
$$
l=\left[\varepsilon_n^3\frac{n}{\log n}\right].
$$
We have, 
\begin{align*}
&\e[u(x+Rx_0+S(m));\tau_x>m,x+S(m)\in K_3]\\
&\le \e[V_\beta(x+Rx_0+S(m));\tau_x>m,x+S(m)\in K_3]\\
&=:E_1+E_2,
\end{align*}
where 
$$
E_1=\e[V_\beta(x+Rx_0+S(m));\tau_x>m,|x+S(l)|>\varepsilon_n\sqrt{n}/2]
$$
and 
$$
E_2=\e[V_\beta(x+Rx_0+S(m));\tau_x>m,|x+S(l)|\le \varepsilon_n\sqrt{n}/2, |S(m)-S(l)|>\varepsilon_n\sqrt{n}/2].
$$
Since $V_\beta(y+S(k)){\rm 1}\{\tau_y>k\}$ is a supermartingale,
\begin{align*}
E_1&\le \e[V_\beta(x+Rx_0+S(l));\tau_x>l,|x+S(l)|>\varepsilon_n\sqrt{n}/2]\\
&\le \e[V_\beta(x+Rx_0+S(l));\tau_x>l,|x+S(l)|>\varepsilon_n\sqrt{n}/2,\max_{k\le l}|X(k)|\le t_n]\\
&\hspace{1cm}+\sum_{k=1}^l\e[V_\beta(x+Rx_0+S(l));\tau_x>l,|x+S(l)|>\varepsilon_n\sqrt{n}/2,|X(k)|> t_n]\\
&\le \e[V_\beta(x+Rx_0+S(l));|x+S(l)|>\varepsilon_n\sqrt{n}/2,\max_{k\le l}|X(k)|\le t_n]\\
&\hspace{1cm}+\sum_{k=1}^l\e[V_\beta(x+Rx_0+S(l));\tau_x>l,
|X(k)|> t_n],
\end{align*}
where 
$$
t_n=\varepsilon_n\frac{\sqrt{n}}{\log n}.
$$
Since $V_\beta(y+Rx_0)\le Cu(y+Rx_0)\le C|y+Rx_0|^p$, 
\begin{align*}
&\e[V_\beta(x+Rx_0+S(l));|x+S(l)|>\varepsilon_n\sqrt{n}/2,\max_{k\le l}|X(k)|\le t_n]\\
&\le C|x+Rx_0|^p\pr(\tau_x>l)+
\e[|S(l)|^p;|S(l)|>\varepsilon_n\sqrt{n}/4,\max_{k\le l}|X(k)|\le t_n]\\
&\le C\frac{|x+Rx_0|^p}{l^{p/2}}u(x+Rx_0)+
\e[|S(l)|^p;|S(l)|>\varepsilon_n\sqrt{n}/4,\max_{k\le l}|X(k)|\le t_n].
\end{align*}
By the Fuk-Nagaev inequality, see Lemma 22 in \cite{DW15},
for $z\ge\frac{\varepsilon_n}{4}\sqrt{n}$, 
$$
\pr(|S(l)|>z,\max_{k\le l}|X(k)|\le t_n)
\le 2d\left(\frac{e\sqrt{d}l}{zt_n}\right)^{z/\sqrt{d}t_n}
\le 2d\left(\frac{e\sqrt{d}\varepsilon_n^2\sqrt{n}}{z}\right)^{\frac{\log n}{4\sqrt{d}}}.
$$
This implies that 
\begin{align*}
&\e[|S(l)|^p;|S(l)|>\varepsilon_n\sqrt{n}/4,\max_{k\le l}|X(k)|\le t_n]\\
&=\left(\frac{\varepsilon_n\sqrt{n}}{4}\right)^p
\pr\left((|S(l)|>\frac{\varepsilon_n\sqrt{n}}{4},\max_{k\le l}|X(k)|\le t_n\right)\\
&\hspace{2cm}+p\int_{\frac{\varepsilon_n\sqrt{n}}{4}}^\infty 
z^{p-1}\pr(|S(l)|>z,\max_{k\le l}|X(k)|\le t_n)dz\\
&\le C(d)\left(\varepsilon_n\sqrt{n}\right)^p
(4e\sqrt{d}\varepsilon_n)^{\frac{\log n}{4\sqrt{d}}}
=o(1).
\end{align*}
As a result, uniformly in $x\in K$ such that $|x|=o(\sqrt{l})$,
\begin{equation}
\label{eq:e11}
\e[V_\beta(x+Rx_0+S(l));|x+S(l)|>\varepsilon_n\sqrt{n}/2,\max_{k\le l}|X(k)|\le t_n]
=o(u(x+Rx_0)).
\end{equation}

Using the supermartingale property of $V_\beta$, we get 
\begin{align*}
&\e[V_\beta(x+Rx_0+S(l));\tau_x>l,|X(k)|> t_n]\\
&\le \e[V_\beta(x+Rx_0+S(k));\tau_x>k,|X(k)|> t_n]\\
&\le C\e[u(x+Rx_0+S(k));\tau_x>k,|X(k)|> t_n]. 
\end{align*}
Using \eqref{diff-bound}, we get 
\begin{align*}
&\e[u(x+Rx_0+S(k));\tau_x>k,|X(k)|> t_n]\\
&\le \e[V_\beta(x+Rx_0+S(k-1));\tau_x>k-1]\pr(|X|> t_n)\\
&+ C\e[|x+Rx_0+S(k-1)|^{p-1};\tau_x>k-1]\e[|X|;|X|>t_n]\\
&+\pr(\tau_x> k-1)\e[|X|^p;|X|>t_n]. 
\end{align*}
By the supermartingale property of $V_\beta$,
$$
\e[V_\beta(x+Rx_0+S(k-1));\tau_x>k-1]\pr(|X|> t_n)
\le Cu(x+Rx_0)\pr(|X|> t_n).
$$
Consequently,
\begin{align}
\label{eq:e12}
\nonumber
&\sum_{k=1}^l
\e[V_\beta(x+Rx_0+S(k-1));\tau_x>k-1]\pr(|X|> t_n)\\
&\le Cu(x+Rx_0)l\pr(|X|> t_n)=o(u(x+Rx_0)),
\end{align}
where the last estimate follows from the bound
$$
l\pr(|X|> t_n)\le \frac{l}{t_n^2\log t_n}\e[|X|^2\log|X|]
=O(\varepsilon_n).
$$
Next,
\begin{align*}
&\e[|x+Rx_0+S(k-1)|^{p-1};\tau_x>k-1]\\
&\le |x+Rx_0|^{p-1}\pr(\tau_x>k-1)
+\e[|S(k-1)|^{p-1};\tau_x>k-1]\\
&\le C|x+Rx_0|^{p-1}\frac{u(x+Rx_0)}{k^{p/2}}
+\e[|S(k-1)|^{p-1};\tau_x>k-1].
\end{align*}
We shall consider the cases $p=2$ and $p>2$ separately.

If $p=2$ then 
\begin{align*}
&\sum_{k=1}^l |x+Rx_0|\frac{u(x+Rx_0)}{k}\e[|X|;|X|>t_n]\\
&\le C|x+Rx_0|u(x+Rx_0)\log l \frac{\e[|X|^2\log|X|;|X|>t_n]}{t_n\log t_n}\\
&\le Cu(x+Rx_0)\frac{|x+Rx_0|\log n}{\sqrt{n}}
\frac{\e[|X|^2\log|X|;|X|>t_n]}{\varepsilon_n}
\end{align*}
If $\varepsilon_n$ converges to zero sufficiently slow then 
$$
\frac{\e[|X|^2\log|X|;|X|>t_n]}{\varepsilon_n}\to0
$$
and, consequently,
\begin{equation*}
\sum_{k=1}^l |x+Rx_0|\frac{u(x+Rx_0)}{k}\e[|X|;|X|>t_n]
=o(u(x+Rx_0))
\end{equation*}
uniformly in $x\in K$ such that $|x|\le\frac{\sqrt{n}}{\log n}$.

If $p>2$ then 
\begin{align*}
&\sum_{k=1}^l |x+Rx_0|^{p-1}\frac{u(x+Rx_0)}{k^{p/2}}
\e[|X|;|X|>t_n]\\
&\le C|x+Rx_0|^{p-1}u(x+Rx_0)\frac{\e[|X|^p;|X|>t_n]}
{t_n^{p-1}}\\
&\le Cu(x+Rx_0)\left(\frac{|x+Rx_0|\log n}{\sqrt{n}}\right)^{p-1}
\frac{\e[|X|^p;|X|>t_n]}{\varepsilon_n^{p-1}}.
\end{align*}
This bound implies that, for any $p\ge2$, 
uniformly in $x\in  K$ with $|x|\le \frac{\sqrt n}{\log n}$, 
\begin{equation}
\label{eq:e13}
\sum_{k=1}^l |x+Rx_0|^{p-1}\frac{u(x+Rx_0)}{k^{p/2}}\e[|X|;|X|>t_n]
=o(u(x+Rx_0))
\end{equation}
provided that $\varepsilon_n\to0$ sufficiently slow. 

In the case $p=2$ we also have 
\begin{align*}
\e[|S(k)|;\tau_x>k]
&\le \sqrt{k\log k}\pr(\tau_x>k)
+\frac{1}{\sqrt{k\log k}}\e[|S(k)|^2,\tau_x>k]\\
&\le C\frac{u(x+Rx_0)\log^{1/2}k}{\sqrt{k}}
+\frac{1}{\sqrt{k\log k}}\e[|S(\tau_x\wedge k)|^2].
\end{align*}
Noting that $|S(k)|^2-dn$ is a martingale and using part (a) of the theorem, we get 
$$
\e[|S(\tau_x\wedge k)|^2]=
d\e[\tau_x\wedge k]\le cu(x+Rx_0)\log k.
$$
Consequently,
$$
\e[|S(k)|;\tau_x>k]
\le C\frac{u(x+Rx_0)\log^{1/2}k}{\sqrt{k}}.
$$
Therefore,
\begin{align*}
&\sum_{k=1}^l \e[|S(k-1)|;\tau_x>k-1]\e[|X|;|X|>t_n]\\
&\hspace{1cm}\le Cu(x+Rx_0)\sqrt{l\log l}\frac{\e[|X|^2\log|X|;|X|>t_n]}{t_n\log t_n}\\
&\hspace{1cm}\le Cu(x+Rx_0)\varepsilon_n^{1/2}
=o(u(x+Rx_0)).
\end{align*}
If $p>2$ then, for every $r\in(0,1)$, 
\begin{align*}
\e[|S(k)|^{p-1};\tau_x>k]
&\le k^{p/2-1/2}\pr(\tau_x>k)
+k^{-r/2}\e[|S(k)|^{p-1+r};\tau_x>k]\\
&\le Cu(x+Rx_0)k^{-1/2}+k^{-r/2}\e[\max_{j<\tau_x}|S(j)|^{p-1+r}].
\end{align*}
Applying \eqref{eq:maximum}, we obtain 
$$
\e[|S(k)|^{p-1};\tau_x>k]
\le Cu(x+Rx_0)k^{-r/2}.
$$
Then, summing up over $k$, we get 
\begin{align*}
&\sum_{k=1}^l \e[|S(k-1)|^{p-1};\tau_x>k-1]\e[|X|;|X|>t_n]\\
&\hspace{1cm}\le Cu(x+Rx_0)l^{1-r/2}t_n^{1-p}\e[|X|^p;|X|>t_n]
=o(u(x+Rx_0))
\end{align*}
if we choose $r$ so that $p+r-3>0$. This shows that 
\begin{align}
\label{eq:e14}
\sum_{k=1}^l \e[|S(k-1)|^{p-1};\tau_x>k-1]\e[|X|;|X|>t_n]
=o(u(x+Rx_0))
\end{align}
for all $p\ge2$.

Considering again the case $p=2$, we have by Lemma~\ref{lem:ub1}, 
\begin{align*}
&\sum_{k=1}^l\pr(\tau_x>k-1)\e[|X|^2;|X|>t_n]\\
&\le Cu(x+Rx_0)\log l\e[|X|^2;|X|>t_n]\\
&\le Cu(x+Rx_0)\frac{\log l}{\log t_n}\e[|X|^2\log|X|;|X|>t_n]\\
&\le Cu(x+Rx_0)\e[|X|^2\log|X|;|X|>t_n]=o(u(x+Rx_0)).
\end{align*}
And if $p>2$ then
\begin{align*}
&\sum_{k=1}^l\pr(\tau_x>k-1)\e[|X|^p;|X|>t_n]\\
&\le Cu(x+Rx_0)\e[|X|^p;|X|>t_n]=o(u(x+Rx_0)).
\end{align*}
Therefore,
\begin{equation}
\label{eq:e15}
\sum_{k=1}^l\pr(\tau_x>k-1)\e[|X|^p;|X|>t_n]
=o(u(x+Rx_0))
\end{equation}
for all $p\ge2$. Combining \eqref{eq:e11}---\eqref{eq:e15}, we conclude that 
\begin{align}
\label{eq:e1}
E_1=o(u(x+Rx_0))
\end{align}
uniformly in $x\in K$ such that $|x|\le\sqrt{n}/\log n$.

We now turn to $E_2$. It follows from \eqref{diff-bound} that $u(x+y)\le u(x)+C|y|^p$ for $|x|\le|y|$. Using this bound, we get
\begin{align*}
&E_2\\
&\le C\e[u(x+Rx_0+S(m));\tau_x>m,|x+S(l)|\le \varepsilon_n\sqrt{n}/2, |S(m)-S(l)|>\varepsilon_n\sqrt{n}/2]\\
&\le C\e[V_\beta(x+Rx_0+S(l));\tau_x>l]\pr(|S(m)-S(l)|>\varepsilon_n\sqrt{n}/2)\\
&\hspace{1cm}+C\pr(\tau_x>l)\e[|S(m-l)|^p;|S(m-l)|>\varepsilon_n\sqrt{n}/2]. 
\end{align*}
Combining the supermartingale property of $V_\beta$ with the Chebyshev inequality, we conclude that
\begin{align*}
&\e[V_\beta(x+Rx_0+S(l));\tau_x>l]\pr(|S(m)-S(l)|>\varepsilon_n\sqrt{n}/2)\\
&\le V_\beta(x)\pr(|S(m)-S(l)|>\varepsilon_n\sqrt{n}/2)
=o(u(x+Rx_0)).
\end{align*}
Therefore,
\begin{align*}
E_2\le o(u(x+Rx_0))+ C\pr(\tau_x>l)\e[|S(m-l)|^p;|S(m-l)|>\varepsilon_n\sqrt{n}/2]\\
\le o(u(x+Rx_0))+ C\frac{u(x+Rx_0)}{l^{p/2}}\e[|S(m-l)|^p;|S(m-l)|>\varepsilon_n\sqrt{n}/2].
\end{align*}
So, it remains to show that 
\begin{align}
\label{eq:e21}
\frac{1}{l^{p/2}}\e[|S(m-l)|^p;|S(m-l)|>\varepsilon_n\sqrt{n}/2]\to0.
\end{align}
Fix some $r>2p\sqrt{d}$. Then, by Corollary 23 in \cite{DW15}
with $y=x/r$,
$$
\pr(|S(k)|>x)\le 2d\left(\frac{er\sqrt{d}k}{x^2}\right)^{r/\sqrt{d}}+k\pr(|X|>x/r).
$$
This implies that 
\begin{align*}
&\e[|S(k)|^p;|S(k)|>\varepsilon_n\sqrt{n}/2]\\
&=(\varepsilon_n\sqrt{n}/2)^p\pr(|S(k)|>\varepsilon_n\sqrt{n}/2)
+p\int_{\varepsilon_n\sqrt{n}/2}^\infty z^{p-1}\pr(|S(k)|>z)dz\\
&\le r^pk\e[|X|^p;|X|>\varepsilon_n\sqrt{n}/2]
+C(p,d,r)(\varepsilon_n^2n)^{p/2}\left(\frac{k}{\varepsilon_n^2n}\right)^{r/\sqrt{d}}.
\end{align*}
Consequently,
\begin{align*}
&\frac{1}{l^{p/2}}\e[|S(m-l)|^p;|S(m-l)|>\varepsilon_n\sqrt{n}/2]\\
&\le
r^p\frac{m}{l^{p/2}}\e[|X|^p;|X|>\varepsilon_n\sqrt{n}/2]
+C(p,d,r)\left(\frac{\varepsilon_n^2n}{l}\right)^{p/2}\left(\frac{m}{\varepsilon_n^2n}\right)^{r/\sqrt{d}}.
\end{align*}
Recalling the definitions of $m$ and $l$, we get 
\begin{align*}
\left(\frac{\varepsilon_n^2n}{l}\right)^{p/2}\left(\frac{m}{\varepsilon_n^2n}\right)^{r/\sqrt{d}}
\le C\varepsilon_n^{-p/2}\log^{p/2}n(\varepsilon_n^2\log^{1/4}n)^{-r/\sqrt{d}}=o(1)
\end{align*}
wenn $\varepsilon_n$ decreases to zero sufficiently slow.

If $p=2$ then, by the Chebyshev inequality,
\begin{align*}
\frac{m}{l^{p/2}}\e[|X|^p;|X|>\varepsilon_n\sqrt{n}/2]
&\le C \log^{3/4}n \varepsilon_n^{-3}\e[|X|^2;|X|>\varepsilon_n\sqrt{n}/2]\\
&\le C \log^{-1/4}n \varepsilon_n^{-3}\e[|X|^2\log |X|;|X|>\varepsilon_n\sqrt{n}/2]=o(1).
\end{align*}
The case $p>2$ is even simpler.
Therefore, $E_2=o(u(x+Rx_0))$. Combining this with~\eqref{eq:e1} 
we obtain~\eqref{eq:K1_K3}. Thus, the  proof is finished. 
\section{Proof of Proposition~\ref{prop:example2}}
By the optional stopping theorem for the martingale $u(x+S(n))$,
\begin{align*}
u(x)&=\e[u(x+S(\tau_x\wedge k)]\\
&=\e[u(x+S(k));\tau_x>k]+\e[u(x+S(\tau_x));\tau_x\le k],
\quad k\ge1.
\end{align*}
Thus, for all $n>m\ge1$,
\begin{align}
\label{ex.6}
\nonumber
E_n-E_m
&=\e[u(x+S(n));\tau_x>n]-\e[u(x+S(m));\tau_x>m]\\
\nonumber
&=\e[-u(x+S(\tau_x));\tau_x\in(m,n]]\\
&:=\Sigma_1+\Sigma_2,
\end{align}
where
$$
\Sigma_1=\sum_{k=m}^{n-1}\sum_{y\in K\cap\{y:|y|\le A\sqrt{n}\}}
\pr(x+S(k)=y,\tau_x>k)\e[-u(y+X);y+X\notin K]
$$
and
$$
\Sigma_2=\sum_{k=m}^{n-1}\sum_{y\in K\cap\{y:|y|> A\sqrt{n}\}}
\pr(x+S(k)=y,\tau_x>k)\e[-u(y+X);y+X\notin\ K].
$$
It is immediate from this representation that the sequence $E_n$ is increasing.

We start with an upper bound for $\Sigma_2$. We first notice that for every $y\in K$ one has  
\begin{align*}
\e[-u(y+X);y+X\notin K]
&=\e[(y_2+X_2)^2-y_1^2;X_2>y_1-y_2]\\
&\le\e[2y_2^2+2X_2^2-y_1^2;X_2>y_1-y_2]\\
&\le |y|^2\pr(X_2>y_1-y_2)+2\e[X_2^2]\\
&\le |y|^2\pr(\tau_y=1)+1.
\end{align*}
Therefore,
\begin{align}
\label{ex.7}
\nonumber
\Sigma_2
&\le \sum_{k=m}^{n-1}\sum_{y\in K\cap\{y:|y|> A\sqrt{n}\}}
\pr(x+S(k)=y,\tau_x>k)\left(|y|^2\pr(\tau_y=1)+1\right)\\
\nonumber
&\le \e\left[|x+S(\tau_x-1)|^2;\tau_x\in(m,n],|x+S(\tau_x-1)|\ge A\sqrt{n}\right]\\
&\hspace{1cm}+
\sum_{k=m}^{n-1}\pr\left(|x+S(k)|\ge A\sqrt{n},\tau_x>k\right).
\end{align}
For every $k<\tau_x$ one has the inequality
$$
|x_2+S_2(k)|<x_1+S_1(k).
$$
This implies that
$$
\max_{k<\tau_x}|x+S(k)|^q 
\le 2^{q/2}\max_{k<\tau_x}|x_1+S_1(k)|^q
$$
for every $q>2$.
Then, by the Burkholder-Davis-Gundy inequality,
\begin{align*}
\e\left[\max_{k<\tau_x}|x+S(k)|^q;\tau_x\le n\right] 
&\le 2^{q/2}\e\left[\max_{k<\tau_x}|x_1+S_1(k)|^q;\tau_x\le n\right]\\
&\le 2^{3q/2}|x|^q+2^{3q/2}\e\left[\max_{k<\tau_x}|S_1(k)|^q;\tau_x\le n\right]\\
&\le C(|x|^q+\e[(\tau_x\wedge n)^{q/2}]).
\end{align*}
According to Lemma~\ref{lem:ub2},
\begin{equation}
\label{ex.8}
\pr(\tau_x>k)\le C(x)\frac{E_k}{k}\le C(x)\frac{E_n}{k},
\quad k\le n.
\end{equation}
This implies that 
$$
\e[(\tau_x\wedge n)^{q/2}]\le C(x)E_n n^{q/2-1}.
$$
As a result we have 
\begin{equation}
\label{ex.9}
\e\left[\max_{k<\tau_x}|x+S(k)|^q;\tau_x\le n\right] 
\le C(x) E_n n^{q/2-1}.
\end{equation}
By the same arguments,
\begin{equation}
\label{ex.10}
\e\left[|x+S(n)|^q;\tau_x>n\right] 
\le C(x)E_n n^{q/2-1}.
\end{equation}
Using the Markov inequality and taking into account \eqref{ex.9},
we obtain 
\begin{align}
\label{ex.11}
\nonumber
&\e\left[|x+S(\tau_x-1)|^2;\tau_x\in(m,n],|x+S(\tau_x-1)|\ge A\sqrt{n}\right]\\
&\hspace{3cm}\le (A\sqrt{n})^{2-q}
\e\left[\max_{k<\tau_x}|x+S(k)|^q;\tau_x\le n\right]
\le \frac{C(x)}{A^{q-2}}E_n.
\end{align}
Moreover, in view of \eqref{ex.10}, for all $k\le n$,
\begin{align*}
\pr\left(|x+S(k)|\ge A\sqrt{n},\tau_x>k\right)
\le (A\sqrt{n})^{-q}\e\left[|x+S(k)|^q;\tau_x>k\right]
\le \frac{C(x)}{A^q}\frac{E_n}{n}.
\end{align*}
Consequently,
\begin{equation}
\label{ex.12}
\sum_{k=m}^{n-1}\pr\left(|x+S(k)|\ge A\sqrt{n},\tau_x>k\right)
\le \frac{C(x)}{A^q} E_n.
\end{equation}
Plugging \eqref{ex.11} and \eqref{ex.12} into \eqref{ex.7},
we have 
\begin{equation}
\label{ex.13}
\Sigma_2\le \frac{C(x)}{A^{q-2}}E_n.
\end{equation}

To estimate $\Sigma_1$ we derive an upper bound for the local probability $\pr(x+S(n)=y,\tau_x>n)$. Set $m=[3n/4]$. Then 
\begin{align*}
&\pr(x+S(n)=y;\tau_x>n)\\
&\hspace{1cm}=\sum_{z\in K}\pr(x+S(m)=z;\tau_x>m)\pr(z+S(n-m)=y;\tau_z>n-m)\\
&\hspace{1cm}=\sum_{z\in K}\pr(x+S(m)=z;\tau_x>m)\pr(y-S(n-m)=z;\tau'_y>n-m)\\
&\hspace{1cm}\le \pr(\tau'_y>n-m)\max_{z\in K}\pr(x+S(m)=z;\tau_x>m),
\end{align*}
where 
$$
\tau'_y=\inf\{k\ge1:y-S(k)\notin K\}.
$$
Replacing $K$ by a half-plane and using  
\cite[Lemma 3]{DW16}, one has 
\begin{equation}
\label{ex.14}
\pr(\tau'_y>n-m)\le\frac{Cd(y)}{\sqrt{n-m}}
\le\frac{C(y_1-|y_2|)}{\sqrt{n}}.
\end{equation}
Furthermore, by Lemma 27 in \cite{DW15},
$$
\max_{z\in K}\pr(x+S(m)=z;\tau_x>m)
\le\frac{C}{m}\pr(\tau_x>m/2).
$$
Taking now into account \eqref{ex.8}, we have
$$
\max_{z\in K}\pr(x+S(m)=z;\tau_x>m)\le C\frac{E_n}{n^2}.
$$
Combining this with \eqref{ex.14}, we arrive at the bound 
\begin{equation}
\label{ex.15}
\pr(x+S(n)=y;\tau_x>n)
\le C(y_1-|y_2|)\frac{E_n}{n^{5/2}},\quad y\in K.
\end{equation}

For every $y\in K$ with $|y|\le A\sqrt{n}$ we have 
\begin{align*}
\e[-u(y+X);y+X\notin K]
&=\e[(y_2+X_2)^2-y_1^2;X_2>y_1-y_2]\\
&\le 2y_2\e[X_2;X_2>y_1-y_2]+\e[X_2^2;X_2>y_1-y_2]\\
&\le 2A\sqrt{n}\e[X_2;X_2>y_1-y_2]+\e[X_2^2;X_2>y_1-y_2].
\end{align*}
Combining this with \eqref{ex.15}, we get 
\begin{align*}
\Sigma_1
&\le
\sum_{k=m}^{n-1}\frac{CE_k}{k^{5/2}}
\sum_{y\in K\cap\{y:|y|\le A\sqrt{n}\}}
(y_1-|y_2|)\left(2A\sqrt{n}E[X_2;X_2>y_1-y_2]+\e[X_2^2;X_2>y_1-y_2]\right)\\
&\le C\frac{E_n}{m^{3/2}}\sum_{y\in K\cap\{y:|y|\le A\sqrt{n}\}}
(y_1-|y_2|)\left(2A\sqrt{n}E[X_2;X_2>y_1-y_2]+\e[X_2^2;X_2>y_1-y_2]\right)\\
&\le C\frac{E_n}{m^{3/2}}A\sqrt{n}
\sum_{y\in K\cap\{y:|y|\le A\sqrt{n}\}}\e[X_2^2;X_2>y_1-y_2].
\end{align*}
The existence of the second moment implies that the sum is $o(n)$.
Therefore,
\begin{equation}
\label{ex.16}
\Sigma_1=o\left(\frac{E_n n^{3/2}}{m^{3/2}}\right).
\end{equation}
Applying this estimate and \eqref{ex.13} to \eqref{ex.6},
we conclude that, if $m\sim cn$ for some $c\in(0,1]$ then
$$
\limsup_{n\to\infty}\left|\frac{E_m}{E_n}-1\right|
\le \frac{C(x)}{A^{q-2}}.
$$
Letting now $A\to\infty$ we finish the proof of (a).

To prove the part (b) we choose $m=m(n)$ such that $\frac{m}{n}\to0$ and $E_m\sim E_n$. 
Then, by the Markov property,
\begin{align}
\label{ex.16a}
\pr(\tau_x>n)
&=\sum_{y\in K}\pr(x+S(m)=y,\tau_x>m)\pr(\tau_y>n-m).
\end{align}
Set 
\begin{align*}
&K_1:=\{y\in K:|y|\le A\sqrt{m},d(y)\le\varepsilon\sqrt{m}\},\\
&K_2:=\{y\in K:|y|\le A\sqrt{m},d(y)>\varepsilon\sqrt{m}\},\\
&K_3:=\{y\in K:|y|> A\sqrt{m}\}.
\end{align*}
By the Brownian approximation, if $\frac{m}{n}\to0$ sufficiently slow,
$$
\pr(\tau_y>n-m)\sim\varkappa\frac{u(y)}{n}
$$
uniformly in $y\in K_2$. Therefore, 
\begin{align}
\label{ex.17} 
\nonumber
&\sum_{y\in K_2}\pr(x+S(m)=y,\tau_x>m)\pr(\tau_y>n-m)\\
&\hspace{1cm}
=\frac{\varkappa}{n}\e[u(x+S(m));y+S(m)\in K_2,\tau_x>m]
+o\left(\frac{E_n}{n}\right).
\end{align}
Applying \eqref{ex.8} and using the Optional Stopping Theorem, 
we obtain 
\begin{align*}
&\sum_{y\in K_3}\pr(x+S(m)=y,\tau_x>m)\pr(\tau_y>n-m)\\
&\hspace{1cm}
\le\frac{C}{n}\sum_{y\in K_3}\pr(x+S(m)=y,\tau_x>m)
\e[u(y+S(n-m));\tau_y>n-m]\\
&\hspace{1cm}
=\frac{C}{n}\sum_{y\in K_3}\pr(x+S(m)=y,\tau_x>m)
\left(u(y)-\e[u(y+S(\tau_y));\tau_y\le n-m]\right)\\
&\hspace{1cm}
\le\frac{C}{n}\e[u(x+S(m));|x+S(m)|>A\sqrt{m},\tau_x>m]\\ 
&\hspace{1.5cm}-\frac{C}{n}\e[u(x+S(\tau_x));\tau_x\in(m,n]].
\end{align*}
It follows from the part (a) that the last expectation is $o(E_n)$.
Consequently,
\begin{align}
\label{ex.18}
\nonumber
&\sum_{y\in K_3}\pr(x+S(m)=y,\tau_x>m)\pr(\tau_y>n-m)\\
&\hspace{1cm}\le 
\frac{C}{n}\e[u(x+S(m));|x+S(m)|>A\sqrt{m},\tau_x>m]
+o\left(\frac{E_n}{n}\right).
\end{align}
By the same arguments,
\begin{align}
\label{ex.19}
\nonumber
&\sum_{y\in K_1}\pr(x+S(m)=y,\tau_x>m)\pr(\tau_y>n-m)\\
&\hspace{1cm}\le 
\frac{C}{n}\e[u(x+S(m));x+S(m)\in K_1,\tau_x>m]
+o\left(\frac{E_n}{n}\right).
\end{align}
Putting together \eqref{ex.16a}---\eqref{ex.19}, we get 
\begin{align}
\label{ex.20}
\nonumber
&\left|\pr(\tau_x>n)-\varkappa\frac{E_n}{n}\right|\\
&\hspace{1cm}\le \frac{C}{n}\e[u(x+S(m));x+S(m)\in K_1\cup K_3,\tau_x>m]
+o\left(\frac{E_n}{n}\right).
\end{align}
If $y\in K_1$ then $u(y)\le C\varepsilon A m$. From this observation and from \eqref{ex.8} we get 
$$
\e[u(x+S(m));x+S(m)\in K_1,\tau_x>m]
\le C\varepsilon A m\pr(\tau_x>m)\le C\varepsilon A E_n.
$$
Furthermore, it follows from \eqref{ex.10} that 
\begin{align*}
&\e[u(x+S(m));x+S(m)\in K_3,\tau_x>m]\\
&\hspace{1cm}\le C(A\sqrt{m})^{2-q}\e[|x+S(m)|^q;\tau_x>m]
\le \frac{C}{A^{q-2}}E_n.
\end{align*}
As a result,
$$
\e[u(x+S(m));x+S(m)\in K_1\cup K_3,\tau_x>m]
\le C\left(A\varepsilon+A^{2-q}\right)E_n
$$
Plugging this into \eqref{ex.20}, we have 
\begin{align*}
\left|\pr(\tau_x>n)-\varkappa\frac{E_n}{n}\right|\le
C\left(A\varepsilon+A^{2-q}\right)\frac{E_n}{n}
+o\left(\frac{E_n}{n}\right).
\end{align*}
Letting first $\varepsilon\to0$ and then $A\to\infty$, we arrive at the desired relation.

If $\e[W_1^2\log(1+|W_1|)]<\infty$ then the finiteness of $\lim_{n\to\infty}E_n$ follows from Theorem~\ref{thm:C2}. Thus, it remains to show that $E_n\to\infty$ for walks with $\e[W_1^2\log(1+|W_1|)]=\infty$.               
It follows from \eqref{ex.6} that 
\begin{align*}
E_{2n}-E_n
&=\e[-u(x+S(\tau_x));\tau_x\in(n,2n]]\\
&=\sum_{k=n}^{2n-1}\int_K\pr(x+S(k)\in dy,\tau_x>k)\e[-u(y+X);y+X\notin K].
\end{align*}
Repeating the arguments from the proof of~\eqref{ex.5a} and taking into account
\eqref{ex.5b}, we get 
$$
E_{2n}-E_n\ge \frac{1}{4}\e[W_1^2;W_1>\sqrt{2}n^2]
\sum_{k=n}^{2n-1}\pr(\tau_x>k).
$$
Using parts (a) and (b), we conclude that there exists $n_0$ such that
$$
E_{2n}-E_n\ge\frac{\varkappa}{8}\e[W_1^2;W_1>\sqrt{2}n^2]E_n,\quad n\ge n_0.
$$
In other words,
$$
E_{2n}\ge\left(1+\frac{\varkappa}{8}\e[W_1^2;W_1>\sqrt{2}n^2]\right)E_n,
\quad n\ge n_0.
$$
Therefore,
$$
\lim_{n\to\infty}E_n\ge 
\prod_{m=0}^\infty\left(1+
\frac{\varkappa}{8}\e[W_1^2;W_1>\sqrt{2}n_0^22^{2m}]\right).
$$
But the product on the right hand side is infinite if $\e[W_1^2\log(1+|W_1|)]=\infty$. This completes the proof of the proposition.

\appendix 

\section{Construction of a positive supermartingale for a random  walk on the positive half line}\label{sec:super_iid} 
Here we will construct explicitly a   positive supermartingale for a one-dimensional 
random  walk on a positive half-line. 
This construction has already appeared in~\cite{PW21}. 

Let $X,X_1,X_2,\ldots $ be i.i.d.\ random variables with the distribution
function $F$, zero mean $\e[X_1]=0$ and finite variance $\sigma^2:=\e[X_1^2]<\infty$. 
Consider the random walk $S_0=0$ and
\[
S_n=X_1+\cdots+X_n,\quad  n\ge 1.
\]
Let
\[
\tau_x:=\inf\{n\ge 1: x+S_n<0\}.
\]
Fix some positive constants $A$, $R$ and define functions for $x\ge 0$,
\begin{align*}
  \beta(x)&=\pr(X\le -x) = F(-x),\\ 
   \beta^I(x)& =\e[{(x+X_1)}^-] = \int_x^\infty \beta(y)dy,\\
\beta^{II}(x)&= \int_x^\infty \beta^I(y)dy,\quad m(x)=A\int_0^x   \beta^{II}(x) dy
\end{align*}
and
\begin{equation}
  \label{eq:defn.v.positive.half.line}
V(x) = \begin{cases}
          x+R+m(x),&x\ge 0\\
          0,&\mbox{otherwise}.
       \end{cases}
\end{equation}
\begin{lemma}\label{lem:drift_positive_half_line}
  Let $\e[X_1]=0$ and $\e[X_1^2]<\infty$. Then,
  for the function $V$ defined in~\eqref{eq:defn.v.positive.half.line}, 
  \begin{equation}
    \label{eq:drift.positive.half.line}
    \e[V(x+X_1);\tau_x>1]\le  V(x)-\beta(x), \quad x\ge 0.
  \end{equation}
\end{lemma}  
\begin{proof}
  Put
  \[
  \Delta(x):=\e [V(x+X_1);\tau_x>1]- V(x). 
  \]
  Then, 
\begin{align*}
 \Delta(x)&=\e[x+X_1,X_1>-x]-x-R\pr(X_1\le-x)+\e[m(x+X_1),X_1>-x]-m(x)\\
  &=-\e[x+X_1,X_1\le -x] -R\beta(x)-m(x)\beta(x)+\int_{-x}^\infty F(dy)(m(x+y)-m(x))\\
  &=\beta^{I}(x)-R\beta(x)-m(x)\beta(x)\\
  &\hspace{1cm}+\int_{-x}^0F(dy)(m(x+y)-m(x))+\int_{0}^\infty F(dy)(m(x+y)-m(x)).
\end{align*}
Integrating twice the integrals by parts we obtain,
\begin{align*}
  \Delta(x)&=\beta^{I}(x)-R\beta(x)-\int_{-x}^0dy F(y)m'(x+y)+\int_{0}^\infty dy \overline F(y)m'(x+y)\\  
  &=\beta^{I}(x)-R\beta(x)-m'(x) \beta^I(0)+A\beta^I(x)\sigma_-^2/2\\
&+\int_{-x}^0dy \beta^I(-y)m''(x+y)
  +m'(x)\overline F^I(0)+\int_{0}^\infty dy \overline F^I(y) m''(x+y),
\end{align*}
where $\sigma_-^2 = \e[{(X_1^-)}^2] $. 
Since $\e[X_1]=0$,
\[
\beta^I(0)-\overline F^I(0)=0.
\]
Therefore,
\begin{align*}
&\Delta(x)=\beta^{I}(x)-R\beta(x)+A\beta^I(x)\frac{\sigma_-^2}{2}\\
&\hspace{1cm}+\int_{0}^\infty dy \overline F^I(y) m''(x+y)+\int_{-x}^0dy \beta^I(-y)m''(x+y).
\end{align*}
We can estimate the second integral as follows,
\begin{align*}
\int_{-x}^0dy \beta^I(-y)m''(x+y) 
&=-A\int_0^{x}dy \beta^I(y)\beta^I(x-y)\\
&=-2A \int_0^{x/2}dy \beta^I(y)\beta^I(x-y)\\
&\le -2A \beta^I(x)\left(\beta^{II}(0)-\beta^{II}(x/2)\right)\\
&=-A\sigma_{-}^2\beta^I(x)+2A\beta^I(x)\beta^{II}(x/2).
\end{align*}
Noting now that $m''$ is negative and, consequently, $\int_{0}^\infty dy \overline F^I(y) m''(x+y)\le0$,
we infer that
\[
\Delta(x)\le \beta^{I}(x)-R\beta(x) -\frac{A}{2}\sigma_-^2\beta^I(x)+2A\beta^I(x) \beta^{II}(x).
\]
Choosing now
\[
A=\frac{4}{\sigma_-^2}
\]
we arrive at the inequality
\begin{equation}\label{hF:eq1}
\Delta(x)\le 2A\beta^I(x) \beta^{II}(x)-\beta^I(x)-R\beta(x)
\end{equation}
The finiteness of $\mathbf{E}{(X_1^{-})}^2$ implies that $\beta^{II}(x)\to0$ as $x\to\infty$.
In particular, there exists $x_0$ such that $2A\beta^I(x) \beta^{II}(x)-\beta^I(x)\le 0$
for all $x\ge x_0$. Increasing $R$, we can make the right hand side in~\eqref{hF:eq1}
negative for $x<x_0$. Increasing $R$ by $1$ we can achieve that
$\Delta(x)\le -\beta(x)$ for all $x\ge0$.
\end{proof}
Lemma~\ref{lem:drift_positive_half_line} immediately implies that 
the sequence $V(x+S_n){\mathbb I}\{\tau_x>n\}$ is a positive supermartingale. As a consequence,
\[
\mathbf{E}[V(x+S_n);\tau_x>n]\le V(x),\quad n\ge1.
\]
Since $V(y)>y$ for all $y>0$, we also have the inequality
\begin{equation}
\label{hF:eq2}
\mathbf{E}[x+S_n;\tau_x>n]\le V(x),\quad n\ge1.
\end{equation}
By the martingale property of $S_n$,
\begin{align*}
x=\mathbf{E}(x+S_{\tau_x\wedge n})
=\mathbf{E}[x+S_{\tau_x};\tau_x\le n]+\mathbf{E}[x+S_n;\tau_x>n].
\end{align*}
From this equality and~\eqref{hF:eq2} we obtain
\[
-\mathbf{E}[x+S_{\tau_x};\tau_x\le n]\le V(x)-x=m(x)+R,\quad n\ge1.
\]
Letting here $n\to\infty$ we finally get the integrability of the overshoot:
\[
-\mathbf{E}[x+S_{\tau_x}]\le m(x)+R.
\]

{\bf Acknowledgment.} 
The authors gratefully acknowledge hospitality of the Saint-Petersburg Department 
of Steklov Institute, where a significant part of this work has been done.  
Special thanks are due to Elena Skripka for the smooth organisation of the visit.


    \end{document}